\documentclass[a4paper,11pt]{article}

\usepackage{amsmath,mathtools,amsthm,amssymb,amsfonts}
\usepackage[numbers,sort&compress]{natbib}
\usepackage[colorlinks,linkcolor=blue,anchorcolor=blue,citecolor=green]{hyperref}
\usepackage{enumerate}
\numberwithin{equation}{section}

\allowdisplaybreaks[4]
\newcommand{\udt}{\frac{\mathrm{d}}{\mathrm{d}t}}
\newtheorem{theorem}{Theorem}[section]

\newtheorem{lemma}{Lemma}[section]

\newtheorem{corollary}{Corollary}[section]
\newtheorem{pro}{Proposition}[section]
\usepackage{mathrsfs}
\makeatletter

\newcommand{\Rmnum}[1]{\expandafter\@slowromancap\romannumeral #1@}
\usepackage{authblk}
\usepackage[title]{appendix}
\allowdisplaybreaks[4]

\makeatother

\begin{document}

\title{Global existence and stability of near-affine solutions of compressible elastodynamics}

\author[a]{Xianpeng Hu}
	\author[b]{Yuanzhi Tu}
  \author[c]{Changyou Wang}
	\author[b]{Huanyao Wen}
	\affil[a]{Department of Applied Mathematics, The Hong Kong Polytechnic University, Hong Kong, China}
  \affil[b]{School of Mathematics, South China University of Technology, Guangzhou, China}
	\affil[c]{Department of Mathematics, Purdue University, West Lafayette, IN 47907, USA}
	%\thanks{Corresponding author.}
  \date{}
	\maketitle
\renewcommand{\thefootnote}{}
	
	\footnote{ {E}-mail: xianphu@polyu.edu.hk(Hu); 202210186864@mail.scut.edu.cn(Tu); wang2482@purdue.edu(Wang)\\; mahywen@scut.edu.cn(Wen).}

\begin{abstract}
    {We prove that for sufficiently small $H^3$-perturbations of an affine solution, the Cauchy problem for 
    the compressible nonlinear elastodynamics in $\mathbb{R}^d$, for $d=2,3$,
    admits a unique global strong solution. Moreover, we establish the asymptotic behavior of the solution.}
\end{abstract}
\vspace{4mm}

{\noindent \textbf{Keywords:} Compressible elastic fluids; near-affine solutions; global existence and\\ stability.}

\vspace{4mm}
\tableofcontents

\section{Introduction}\label{section_1}
We consider the compressible {elastodynamics} in Eulerian coordinate \cite{Gurtin_1981,Dafermos_2010}:
\begin{eqnarray} \label{elasticity_equation_1}
    \begin{cases}
        \partial_t\rho+{\rm div} (\rho u) =0,\\
        \partial_t(\rho u)+{\rm div} (\rho u \otimes u)  + \nabla P(\rho)={\rm div}(\rho F F^\top),\\
        \partial_t F + u\cdot\nabla F=\nabla u F,\\
        (\rho,u,F)|_{t=0} = (\rho_0, u_0, F_0)
    \end{cases}
\end{eqnarray}
on $\mathbb{R}^d \times (0, +\infty)$, {with} $d = 2, 3${, where} $\rho=\rho(x,t)$, $u=u(x,t)$, $F=F(x,t)$ and $P(\rho)=R\rho^{\gamma}$ denotes the fluid density, the velocity field, the deformation gradient and the pressure, respectively. The $i$th component of ${\rm div}(\rho F F^\top)$ on the right-hand side of the momentum equation is $\partial_j(\rho F_{ik} F_{jk})$ and $(\nabla u F)_{ij}=\partial_k u_i F_{kj}$. The physical coefficients $R, \gamma$ satisfy $R>0$ and $\gamma>1$. Without loss of generality, we set $R=1$ throughout this paper. For smooth solutions $(\rho, u, F)$ of \eqref{elasticity_equation_1}, the physical constitution relation
\begin{align}\label{constitution_laws}
\begin{cases}
\rho\  \mathrm{det}F=1,\\
F_{j}\cdot\nabla F_{k}=F_{k}\cdot\nabla F_{j}, \quad 1\leq   j, k\leq d
\end{cases}
\end{align}
holds true for all $t\ge 0$, see \cite{Hu_2011, Qian_2010}. As a consequence of  \eqref{constitution_laws}, it is verified in \cite{Hu_2020, Hu_2020_1} that
\begin{align*}\label{constitution_laws_1}
{\rm div}(\rho F^\top)=0.
\end{align*}

We begin by briefly reviewing the literature on the well-posedness of elastodynamics {around a constant equilibrium}. In three-dimensional compressible elastodynamics, John \cite{John_1984} showed that singularities arise from arbitrarily small, spherically symmetric displacements in the absence of a null condition, while Tahvildar-Zadeh \cite{Tahvildar-Zadeh_1998} established a similar result for large displacements. In \cite{John_1988, Klainerman_1996}, the authors proved the existence of almost global solutions for three-dimensional quasilinear wave equations under the assumption of sufficiently small initial data. Significant progress in demonstrating the global existence of classical solutions was achieved independently by Sideris \cite{Sideris_1996,Sideris_2000} and Agemi \cite{Agemi_2000}, both assuming that the nonlinearity meets a null condition in three dimensions. For incompressible elastodynamics, particularly in 3D, the Hookean component of the system is inherently degenerate and satisfies a null condition, allowing Sideris and Thomases \cite{Sideris_2005,Sideris_2007} to establish global existence for this case. However, as noted by Lei \cite{Lei_2015}, the existence problem becomes more intricate in two dimensions, even with the null condition, since quadratic nonlinearities exhibit only critical time decay. The first significant result in tackling these issues was presented by Lei, Sideris, and Zhou \cite{Lei_2015}, who demonstrated the almost global existence of incompressible isotropic elastodynamics in 2D by reformulating the system in Eulerian coordinate. Lei \cite{Lei_2016} further leveraged the strong linear degeneracy of the nonlinearities in incompressible isotropic Hookean elastodynamics, which automatically satisfy a strong null condition, proving the global existence of classical solutions to the 2D Cauchy problem for small initial displacements in a specific weighted Sobolev space, with an alternative proof available in \cite{Wang_2017}. {All these references studied the stability issue of \eqref{elasticity_equation_1} around a non-zero constant equilibrium $$(\hat{\rho}, \hat{u}, \hat{F})=(1,0, I).$$ Indeed, the linearized system of \eqref{elasticity_equation_1}
at $(1, 0, I)$ is a coupling wave system for shear waves and pressure waves with two different wave speeds, where $I$ stands for the identity matrix of order $d$. For incompressible models, the pressure wave disappears, and the linearized system for incompressible elastodynamics takes the form
\begin{equation}\label{wave}
\partial_t^2 \omega-\Delta \omega=0,
\end{equation}
 where $\omega$ is either the vorticity of velocity $u$ or the curl of $F$. The stability around a constant equilibrium is justified by following the well posedness theories for quasilinear wave equations via either the vector field method \cite{Sideris_2005, Lei_2016} or the space-time resonance method \cite{Wang_2017}.}

{Different from the stability around a constant equilibrium, there is an extensive range of interests on the stability of \eqref{elasticity_equation_1} around some non-constant equilibrium in 
the communities of both physics and mathematics. As an example of non-constant equilibrium, in }\cite{Sideris_2017}, an affine motion is defined as a one-parameter family of deformations of the form
$$x(y,t)=A(t)y,$$ where $x(y,t)$ is the flow map with the initial position $y\in\mathbb{R}^d$. For affine motions, all hydrodynamical quantities are expressed explicitly in terms of the deformation gradient $A(t)$ and the initial conditions. The authors in \cite{Hu_preprint} {studied the stability of} an affine solution for the 3D compressible viscoelasticity system in Eulerian coordinate, given by
\begin{equation}\label{1-1}
    \bar{\rho}=\frac{1}{(1+t)^{3}}, \bar{u} = \frac{x}{1+t},  \bar{F}= (1+t)I,
  \end{equation}
 with the initial data $(1, x, I).$ {The magnitude of the deformation gradient in \eqref{1-1} grows at a rate proportional to time, and that deformation gradient does not lead to finite time collapse or blow-up of the flow. The stability of the affine solution \eqref{1-1} is justified in \cite{Hu_preprint} for compressible viscoelasticity under small $H^2$-perturbations. The key difficulty lies in the fact the coefficients of the linearised system in \cite{Hu_preprint} depends on the spatial and temporal variables. A natural question then arises: does such stability persist for compressible elastodynamics \eqref{elasticity_equation_1} when the viscosity is absent. In this work, we provide an affirmative answer
 by establishing}
 %Our main result is stated as follows:
\begin{theorem}\label{global_existence} For $d=2, 3$,
there exists a diffeomorphism $\xi_0: \mathbb{R}^d \to \mathbb{R}^d$  {and a small constant $\varepsilon>0$} such that if the initial data $(\rho_0, u_0, F_0)$, with
$(\rho_0(x) -1, u_0(x) -x, F_0(x) -I)\in H^3_x(\mathbb{R}^d)$, satisfies
{\begin{eqnarray*}\label{decay_estiwmates_1}
    \begin{cases}
      \rho_0(x) \mathrm{det}F_0(x)=1,\,\,\,\,F_0(x) = \nabla_y \xi_0(y),\,\,\,\,  y =\xi_0^{-1}(x), \ \forall x\in\mathbb{R}^d, \\[2mm]
      \|\rho_0(x) -1\|_{H^3_x} + \|u_0(x) -x\|_{H^3_x} + \|F_0(x) -I\|_{H^3_x} + \|\xi_0^{-1}(x)-x\|_{H^4_x}\le \varepsilon,
    \end{cases}
\end{eqnarray*}}then the Cauchy problem \eqref{elasticity_equation_1} admits a unique global strong solution 
$(\rho, u, F)$ such that $\big(\rho(x,t) -\frac{1}{(1+t)^d}, u(x,t)-\frac{x}{1+t}, F(x,t) -(1+t)I\big) \in C\big([0, \infty); H^3_x(\mathbb{R}^d)\big)$. Moreover, $(\rho, u, F)$ satisfies the following estimates:
\begin{eqnarray*}\label{decay_estimates_1}
  \begin{cases}
    \|\nabla_x^{i}(\rho(x,t) -\frac{1}{(1+t)^d})\|_{L^2_x}\le C\varepsilon(1+t)^{-\frac{d}{2}-\frac{1}{2}-i},\\[1mm]
     \|\nabla_x^{i}\big(u(x,t)-\frac{x}{1+t}, F(x,t) -(1+t)I\big)\|_{L^2_x} \le C\varepsilon(1+t)^{\frac{d}{2}+\frac{1}{2}-i}
  \end{cases}
\end{eqnarray*}for all $t>0$ and $i=0,1,2,3$, {where $C$ is a generic positive constant depending on some known constants but independent of $\varepsilon$ and $t$.}
\end{theorem}
We prefer to work in the framework of $H^3$-spaces in order to keep the strategy as clear as possible. With our strategy in hand, it is possible to reduce the regularity of solutions to some Besov spaces, however the presentation would be more technical and complicated, and we leave it for the interested readers.

\iffalse
{Although $\bar u$ in (\ref{1-1}), the same as that in \cite{Hu_preprint}, can be viewed as a special case within the general class of solutions in \cite{Grassin_1998}, the analysis of the viscoelasticity system faces a fundamental obstacle: how to get the estimate of divergence of $u$? This estimate is successfully achieved in Eulerian coordinate in \cite{Hu_preprint}. However, this approach becomes invalid in the inviscid case. The main difficulty for the inviscid case lies in the loss of control over the divergence of velocity in Eulerian formulation. To overcome it, we take a different route and reformulate the problem in Lagrangian coordinate. In this framework, the flow map gives us better control, and the estimates that fail in the Eulerian setting can be recovered. With these tools, we are able to build the unique global solution and then translate it back into Eulerian coordinate. This change of viewpoint is more than a simple change of variables---it provides the key idea for proving stability without viscosity and shows the wider use of Lagrangian methods in studying compressible flows.}
\fi

{With the absence of the elastic stress $\textrm{div}(\rho FF^\top)$ in the momentum equation, the system \eqref{elasticity_equation_1} becomes a degenerate system whenever the physical vacuum is present. Indeed, by globally solving a Hamiltonian system of ordinary differential equations for the deformation gradient, the free boundary problem of the three dimensional compressible and incompressible Euler equations with an affine initial condition and physical vacuum boundary condition is constructed in \cite{Sideris_2017}. Hadzic and Jang in \cite{Hadzic_2018} has verified the global-in-time nonlinear stability in Sobolev spaces of the affine solutions to compressible Euler equations for $\gamma\in(1,\frac53]$, and the full range of $\gamma\in (1,\infty)$ is studied in \cite{Shkoller_2019}. Note that the determinant of deformation gradients of the affine solution in \cite{Sideris_2017, Hadzic_2018, Shkoller_2019} grows at the maximal rate of $(1+t)^3$, as $t\rightarrow\infty$. In this sense, Theorem \ref{global_existence} would be regarded as the asymptotic nonlinear stability of the affine solution in \cite{Sideris_2017, Hadzic_2018, Shkoller_2019} when the elastic stress is taken into consideration in the momentum equation.

The presence of the elastic stress $\textrm{div}(\rho FF^\top)$ in the momentum equation makes the wave operator non-degenerate. Motivated by the study of compressible Euler equations in \cite{Shkoller_2019}, we turn to employing a Lagrangian transformation and consider a different class of affine solutions that nevertheless share the same scaling property, namely $\det A(t)\sim (1+t)^3$. Compared with the Euler equations, the presence of elastic stress introduces an additional diffusion mechanism, see~\eqref{compressible_elasticity_equation_5}. However, this feature requires a different analytical framework, and the framework developed in \cite{Shkoller_2019} cannot be applied in a straightforward manner. Fortunately, the expanding factor on $t$ in the affine solution \eqref{1-1} contributes a damping term for the velocity in linearized system of \eqref{elasticity_equation_1} in Lagrangian coordinate, see \eqref{compressible_elasticity_equation_5}. However, due to the expanding factor in $t$ of the affine solution,  the coefficients of the linearized system of  \eqref{elasticity_equation_1} involves functions depending on $t$. This kind of
hyperbolic systems with variable coefficients excludes the applicability of both the vector field method and the space-time resonance method, which are designed for quasilinear wave equations with constant variables. In order to overcome this difficulty, we turn to the $t-$dependent damping effect of the velocity. Indeed, with this damping at hand, different weights for unknowns will be constructed in the form of $\mathcal{L}(t)$, and  the differential inequality \eqref{est_L_1} is verified. More precisely, since we are concerned with the stability of a non-constant equilibrium, it is necessary to introduce different temporal weights in a series of uniform estimates. For instance, the terms in $\mathcal{L}(t)$ in \eqref{L} carry distinct weights {depending on $\gamma$ and the dimension}. Thanks to the presence of the deformation gradient, which provides an effective diffusion mechanism, the temporal weights of all derivatives of all {unknowns} become uniform. This uniformity plays a key role in establishing the global well-posedness of the compressible elastodynamics system in both two and three dimensions for all $\gamma>1$ within a unified framework. As a consequence, the global existence in Lagrangian coordinate would be established in Theorem \ref{global_existence_Lagrangian} with the help of }the local well-posedness  and the global-in-time estimate of the quantity $\mathcal{L}(t)$. 

The global well-posedness in Lagrangian coordinate will be transferred back into the Eulerian coordinate as stated in Theorem \ref{global_existence} via the flow map. We remark that in Lagrangian coordinate, the impact of differentiation on time decay is largely concealed, whereas in Eulerian coordinate each additional derivative accelerates decay by a factor of $(1+t)^{-1}$. Moreover, the $L^2$ norms of $\rho(x,t) -\frac{1}{(1+t)^d}$, $u(x,t)-\frac{x}{1+t}$ and  $F(x,t) -(1+t)I$ decay at faster rates compared with those obtained in \cite{Hu_preprint}.

The rest of the paper is organized as follows. In Section \ref{section_3},
we will reformulate the Cauchy problem \eqref{elasticity_equation_1} in the Lagrangian coordinate and establish the global well-posedness of system \eqref{compressible_elasticity_equation_5} through energy estimates. Building upon Theorem \ref{global_existence_Lagrangian}, Section \ref{section_4} is devoted to the construction of the global solution of system \eqref{elasticity_equation_1} 
and the proof of Theorem \ref{global_existence}.
\section{Global well-posedness in the Lagrangian coordinate}\label{section_3}
\subsection{Reformulation in the Lagrangian coordinate}\label{section_2}
In this section, we reformulate the Cauchy problem \eqref{elasticity_equation_1} in the Lagrangian coordinate. Generally speaking, establishing the existence and uniqueness of a global-in-time solution to \eqref{elasticity_equation_1} directly proves to be challenging. To overcome this difficulty, we adopt the Lagrangian coordinate as our analytical framework.  To facilitate the transition between the Lagrangian and the Eulerian coordinate, we assume that there exists a diffeomorphism  $\xi_0$ from $\mathbb{R}^d$ to $\mathbb{R}^d$ satisfying
\begin{equation}\label{flow_map_6}
  x =\xi_0(y),\,\,\,\,\,\,\,\,\,\, F_0(\xi_0(y)) = \nabla \xi_0(y).
\end{equation}
Let $\xi(y,t)$ and $V(y,t)$ denote the ``position" and ``velocity" of the fluid particle $y$ at time $t$. Then we have
\begin{align*}\label{flow_map}
    \begin{cases}
        \partial_t \xi(y,t)=u(\xi(y, t),t), \ {\rm{in}}\ \mathbb{R}^d\times (0,\infty),\\
        \xi(y,0)=\xi_0(y), \ {\rm{in}}\ \mathbb{R}^d.
    \end{cases}
    \end{align*}
%on $\mathbb{R}^d\times (0, +\infty)$.
We define the following Lagrangian variables:
\begin{equation*}\label{flow_map_1}
  V(y,t):= u(\xi(y, t),t),\,\pi(y,t):=\rho(\xi(y,t),t),\,p(y,t):=P(\xi(y,t),t),\,H(y,t):=F(\xi(y,t),t)
\end{equation*}
and introduce the quantities
\begin{equation}\label{flow_map_2}
  B:= [(\nabla \xi)^{-1}]^\top,\,\,\,\,\,\,\,\,\,\,\,\,\,\,\,\,\,\,\,\,\,\,\,\,\,\,\,\,\,\,\,\,\mathcal{J}:= {\rm det} \nabla \xi,\,\,\,\,\,\,\,\,\,\,\,\,\,\,\,\,\,\,\,\,\,\,\,\,\,\,\,\,\,\,\,\,b:= \mathcal{J}B.
\end{equation}
These definitions yield the fundamental relationship
\begin{eqnarray}\label{flow_map_2_1}
  \begin{cases}
  \partial_{\xi_i} \rho(\xi(y,t),t) = B_{ik} \partial_{y_k} \pi(y,t),\\ \partial_{\xi_i} u_{j}(\xi(y,t),t) = B_{ik} \partial_{y_k} V_{j}(y,t),\\ \partial_{\xi_k} F_{ij}(\xi(y,t),t) = B_{ks} \partial_{y_s} H_{ij}(y,t).
  \end{cases}
\end{eqnarray}
Using \eqref{flow_map_2_1}, we can reformulate the system \eqref{elasticity_equation_1} in the  Lagrangian coordinate as
\begin{eqnarray}\label{compressible_elasticity_equation_2}
	\begin{cases}
		\partial_t \xi=V,\\
        \partial_t \pi + \pi B_{kj}\partial_j V_{k}=0,\\
		\pi \partial_t V_{i} + B_{il}\partial_l p - B_{kl}\partial_l(\pi H_{ij}H_{kj}) =0,\\
        \partial_t H_{ij} = B_{kl}\partial_l V_{i} H_{kj},\\
		(\xi, V, \pi, H)\big|_{t=0} = \big(\xi_0, V_0, \pi_0, H_0\big),
	\end{cases}
  \end{eqnarray}
on $\mathbb{R}^d\times (0, +\infty)$, where $V_0 = u_0(\xi_0),\,\, \pi_0 = \rho_0(\xi_0),\,\, H_0 = F_0(\xi_0).$

Building upon the approaches developed in \cite{Gu_2023, Xu_2013}, we simplify system \eqref{compressible_elasticity_equation_2} as follows. First, applying Jacobi's formula
\begin{equation}\label{flow_map_3}
    \mathcal{J}_t=\mathcal{J} B_{kj}\partial_j V_{k} = b_{kj}\partial_j V_{k},
\end{equation}
we reduce the second equation in \eqref{compressible_elasticity_equation_2} to the conservation law $\partial_t (\pi\mathcal{J})=0$. This implies
\begin{equation}\label{flow_map_4}
    \pi =\mathcal{J}^{-1} \pi_{0}\mathcal{J}_{0}=\mathcal{J}^{-1},
\end{equation}
where we have used the fact that $\pi_{0}\mathcal{J}_{0}=\rho_0\,{\rm det} F_0= 1$.

Since $\partial_t(B^{\top}\nabla \xi) = 0$ holds, we can deduce
\begin{equation}\label{flow_map_4_1}
  \partial_t B_{ij} = -B_{ik} \partial_k V_{l}B_{lj}.
\end{equation}
By combining \eqref{flow_map_4_1} with the fourth equation of \eqref{compressible_elasticity_equation_2}, we deduce that $\partial_t (B^\top H)=0.$ Together with the initial condition $F_0(\xi_0(y)) = \nabla \xi_0(y)$ from \eqref{flow_map_6}, we obtain
\begin{equation}\label{flow_map_5}
  H= \nabla \xi (\nabla \xi_0)^{-1} \nabla \xi_0 = \nabla \xi.
\end{equation}
Furthermore, applying differential identities similar to \eqref{flow_map_3} and \eqref{flow_map_4_1}, we derive
\begin{equation}\label{flow_map_5_1}
  \nabla\mathcal{J}=\mathcal{J} B_{kj}\partial_j \nabla \xi_{k},\,\,\,\,\,  \nabla B_{ij} = -B_{ik} \partial_k \nabla \xi_{l}B_{lj},
\end{equation}
which yields the Piola identity:
\begin{equation}\label{flow_map_7}
  \partial_l b_{il} = \partial_l (\mathcal{J} B_{il})=0.
\end{equation}
Combining \eqref{flow_map_5}, \eqref{flow_map_7} with the fact that $B^{\top}\nabla \xi = I$, we can simplify the third equation in \eqref{compressible_elasticity_equation_2} as follows.
\begin{equation}\label{flow_map_8}
  \begin{split}
  B_{kl}\partial_l(\pi H_{ij}H_{kj}) &= \mathcal{J}^{-1}b_{kl}\partial_l(\pi H_{ij}H_{kj}) = \mathcal{J}^{-1}\partial_l(b_{kl}\pi H_{ij}H_{kj})\\ &= \mathcal{J}^{-1}\partial_l(B_{kl} H_{ij}H_{kj}) = \mathcal{J}^{-1}\partial_l(\delta_{lj}H_{ij}) = \mathcal{J}^{-1}\Delta \xi_{i}.
  \end{split}
\end{equation}
Finally, substituting \eqref{flow_map_4} -\eqref{flow_map_8} into \eqref{compressible_elasticity_equation_2}, we derive the following simplified system:
\begin{eqnarray} \label{compressible_elasticity_equation_3}
	\begin{cases}
		\partial_t \xi=V,\\
		\partial_t V_{i} + b_{il}\partial_l p -\Delta \xi_{i} =0,\\
		(V,\xi)|_{t=0} = \big(V_0, \xi_0\big)
	\end{cases}
  \end{eqnarray}
on $\mathbb{R}^d\times (0, +\infty)$.
It is not hard to see that, under the condition that $\xi(y,t)$ is invertible for all $y\in \mathbb{R}^3$ and $t\ge 0$, the solutions to system \eqref{elasticity_equation_1} and system \eqref{compressible_elasticity_equation_3} are equivalent. In this section, the Sobolev norms of spatial variables are with respect to the variable $y.$
Our main result in this section is stated as follows:
    \begin{theorem}\label{global_existence_Lagrangian}
       For $d=2,3$,  there exists a sufficiently small  constant $\varepsilon_1>0$ such that the Cauchy problem \eqref{compressible_elasticity_equation_3} admits a global strong solution 
       $(V,\xi)$ satisfying $(V(y,t)-y, \xi(y,t) - (1+t)y)$ in $C\big([0, \infty); H^3_y(\mathbb{R}^d)\times H^4_y(\mathbb{R}^d)\big)$,  provided that
          \begin{equation*} \label{initial_assumption_Lagrangian}
              \begin{split}
              \|V_0(y)-y\|_{H^3_y} + \|\xi_0(y) -y\|_{H^4_y}\le \varepsilon_1.
              \end{split}
          \end{equation*}
Moreover, $(V,\xi)$ satisfies the following estimates:
\begin{eqnarray}\label{decay_estimates_Lagrangian}
  \begin{cases}
    \|V(y,t) -y\|_{H^3_y} \le C\varepsilon_1(1+t)^{\frac{1}{2}},\\[1mm]
    \|\xi(y,t) - (1+t)y\|_{H^3_y} \le C\varepsilon_1(1+t)^{\frac{3}{2}},\\[1mm]
     \|\nabla_y \xi(y,t) - (1+t)I\|_{H^3_y} \le C\varepsilon_1(1+t)^{\frac{1}{2}}
  \end{cases}
\end{eqnarray}
for all $t>0$, where $C$ is a generic positive constant depending on some known constants but independent of $\varepsilon_1$ and $t$.
\end{theorem}
The proof of Theorem \ref{global_existence_Lagrangian} will be given in Sections \ref{subsection_3}-\ref{subsection_6}.
\subsection{Perturbations of affine solutions}\label{subsection_2}
In order to simplify the proof of Theorem \ref{global_existence_Lagrangian}, we reformulate the system \eqref{compressible_elasticity_equation_3} to the following system \eqref{compressible_elasticity_equation_5}. Motivated by \cite{Shkoller_2019}, we consider the affine solution given by $\xi(y,t)=(1+t)y$ and denote the perturbation to this affine flow by $\eta(y, t)$. We then seek solutions to \eqref{compressible_elasticity_equation_3} of the form
\begin{equation}\label{perturbation}
  \xi(y,t)=(1+t)\eta(y,t).
\end{equation}
The velocity $v$  associated with the perturbation  $\eta$ is defined as $v(y, t): = \partial_t \eta(y, t).$ From this definition, it follows that
\begin{equation}\label{perturbation_1}
v = \partial_t \eta = \frac{1}{1+t}\partial_t \xi - \frac{1}{(1+t)^2}\xi = \frac{1}{1+t} V - \frac{1}{(1+t)^2}\xi.
\end{equation}
We define the following quantities as in \eqref{flow_map_2}:
\begin{equation}\label{perturbation_1_1}
  J:={\rm det}\nabla \eta,\,\,\,\,\,\,\,\,\,\,\,\,\,\,\,\,\,\,\,\,\,A:=[(\nabla \eta)^{-1}]^{\top},\,\,\,\,\,\,\,\,\,\,\,\,\,\,\,\,\,\,\,\,\,a:=JA,\,\,\,\,\,\,\,\,\,\,\,\,\,\,\,\,\,\,\,\,\,q:=J^{-\gamma}.
\end{equation}
Combining \eqref{flow_map_2} with \eqref{perturbation} and \eqref{perturbation_1_1}, we obtain
\begin{eqnarray}\label{perturbation_2}
	\begin{cases}
		J= \frac{1}{(1+t)^d} \mathcal{J},\\
		A= (1+t)B,\\
    a= \frac{1}{(1+t)^{d-1}} b,\\
    q= (1+t)^{d\gamma} p.
	\end{cases}
  \end{eqnarray}
Substituting \eqref{perturbation}-\eqref{perturbation_2} into \eqref{compressible_elasticity_equation_3}, we obtain the following system:
\begin{eqnarray} \label{compressible_elasticity_equation_4}
	\begin{cases}
		\partial_t \eta=v,\\
		\partial_t v_{i} + \frac{2}{1+t} v_{i} + \frac{1}{(1+t)^{d(\gamma -1)+2}} a_{il}\partial_l q -\Delta \eta_{i} =0,\\
		(v,\eta)|_{t=0} = \big(v_0, \eta_{0}\big)
	\end{cases}
  \end{eqnarray}
on $\mathbb{R}^d\times (0, +\infty)$, where $v_0(y) = V_0(y)-\xi_0(y)$, $\eta_{0}(y) = \xi_{0}(y)$. Since the non-constant solution of the system \eqref{elasticity_equation_1} corresponds to the system \eqref{compressible_elasticity_equation_4} with $\overline{v}=0, \overline{\eta}= y$, we define $\widetilde{\eta} = \eta-y$. This yields the following system:
\begin{eqnarray} \label{compressible_elasticity_equation_5}
	\begin{cases}
		\partial_t \widetilde{\eta}=v,\\
		\partial_t v_{i} + \frac{2}{1+t} v_{i} + \frac{1}{(1+t)^{d(\gamma -1)+2}} a_{il}\partial_l q -\Delta \widetilde{\eta}_{i} =0,\\
		(v,\widetilde{\eta})|_{t=0} = \big(v_0, \widetilde{\eta}_0\big)
	\end{cases}
  \end{eqnarray}
  on $\mathbb{R}^d\times (0, +\infty)$.
\subsection{Local existence and uniqueness}\label{subsection_3}
In this subsection, we will establish the local well-posedness of the elasticity system \eqref{compressible_elasticity_equation_5}. Throughout the rest of this section, let C denote a generic positive constant that depends on  $M_2$ and other known constants but is independent of $m$, $T_1$ and $M_1$. The main Proposition of this section is stated as follows:
\begin{pro}\label{local_existence_1}
  Under the assumptions of Theorem \ref{global_existence_Lagrangian}, for any $M_1, M_2>0$ there exists $T_1=T_1(M_1, M_2)>0$ such that if
    \begin{eqnarray*}\label{initial_assumption_5}
            \begin{cases}
            \|v_0(y)\|_{H^3_y}^2 + \|\widetilde{\eta}_0(y)\|_{H^4_y}^2\le M_1,\\
            \frac{1}{M_2}\le{\rm det}[\nabla \widetilde{\eta}_0(y) + I]\le M_2, \ \forall y\in\mathbb{R}^d,
            \end{cases}
          \end{eqnarray*}
then  \eqref{compressible_elasticity_equation_5} admits a unique strong solution  $(v, \widetilde{\eta})(y,t)
\in L^{\infty}([0, T_1]; H^3_y(\mathbb{R}^d)\times H^4_y(\mathbb{R}^d))$.%, where the positive constants $M_1$ and $M_2$ are given by
                \end{pro}
\begin{proof}
  \subsubsection*{Step 1: Construction of approximate solutions}
  For a given $(v_0, \widetilde{\eta}_0)$ in \eqref{local_existence_1}, there exists $(v_0^m, \widetilde{\eta}_0^m)$ in $C_c^\infty(\mathbb{R}^d)$ 
  such that
  \begin{eqnarray*}\label{local_existence_2}
    \begin{cases}
      \|v_0^m - v_0\|_{H^3} + \|\widetilde{\eta}_0^m - \widetilde{\eta}_0\|_{H^4}\,\,\rightarrow 0 \,\,\,\,\,\,\,\,\,\,\,\,\,\,\,\,\,\,as\,\,\,\,\,\,\,\,\,\,\,\,\,\,m\rightarrow\infty,\\
      \|v_0^m\|_{H^3}^2 + \|\widetilde{\eta}_0^m\|_{H^4}^2\le 2M_1 \,\,\,\,\,\,\,\,\,\,\,\,\,\,\,\,\,\,\,\,\,\,\,\,\,\,\,\,\,\,\,\,\,\,\,\,\,\,\,\,\,\,\,\,\,\,\,\,\,\,\,\,\,\,\,\,\,\,\,\,\,\,\,m\ge 1,\\
      \frac{1}{2M_2}\le{\rm det}[\nabla\widetilde{\eta}_0^m +I]\le 2M_2.\,\,\,\,\,\,\,\,\,\,\,\,\,\,\,\,\,\,\,\,\,\,\,\,\,\,\,\,\,\,\,\,\,\,\,\,\,\,\,\,\,\,\,\,\,\,\,\,\,\,\,\,\,\,\,\,\,m\ge 1.
    \end{cases}
  \end{eqnarray*}
We aim to construct local classical solutions $(v^m, \widetilde{\eta}^m)$ satisfying
\begin{eqnarray} \label{local_existence_3}
	\begin{cases}
		\partial_t \widetilde{\eta}^m=v^m,\\
		\partial_t v_{i}^m+ \frac{2}{1+t} v_{i}^m+ \frac{1}{(1+t)^{d(\gamma -1)+2}} a_{il}^m\partial_l q^m-\Delta \widetilde{\eta}_{i}^m=0,\\
		(v^m,\widetilde{\eta}^m)\big|_{t=0} = \big(v_0^m, \widetilde{\eta}_0^m\big)
	\end{cases}
  \end{eqnarray}
  on $\mathbb{R}^d\times [0, T^m]$.

  Substituting  \eqref{perturbation_1_1} and \eqref{perturbation_2} into \eqref{flow_map_2} and \eqref{flow_map_5_1}, we derive
  \begin{equation}\label{local_existence_7_1}
    \begin{split}
      \nabla J^m= a_{ij}^m\partial_j\nabla \widetilde{\eta}_{i}^m,
    \end{split}
  \end{equation}
  which, together with $q^m=(J^m)^{-\gamma}$, implies that the system \eqref{local_existence_3} can be simplified as
  \begin{equation}\label{local_existence_3_1}
    \begin{split}
      \partial_t^2 \widetilde{\eta}_{i}^m+ \frac{2}{1+t} \partial_t \widetilde{\eta}_{i}^m- \frac{\gamma}{(1+t)^{d(\gamma -1)+2}}(J^m)^{-(\gamma +1)} a_{il}^ma_{js}^m\partial_s\partial_l \widetilde{\eta}_{j}^m- \Delta \widetilde{\eta}_{i}^m=0.
    \end{split}
  \end{equation}
We define
\begin{equation}\label{local_existence_3_1_1}
  \begin{split}
    \widetilde{u}: = (1+t) \widetilde{\eta}^m.
  \end{split}
\end{equation}
Combining \eqref{local_existence_3_1_1} with \eqref{perturbation_1_1}, we obtain
\begin{eqnarray} \label{local_existence_3_1_2}
	\begin{cases}
		\eta^m=\widetilde{\eta}^m+ y = \frac{1}{1+t}(\widetilde{u} + (1+t)y),\\
		J^m= {\rm det}\nabla \eta^m = \frac{1}{(1+t)^d}{\rm det}\nabla(\widetilde{u} + (1+t)y),\\
    A^m= [(\nabla \eta^m)^{-1}]^\top = (1+t)[(\nabla(\widetilde{u} + (1+t)y)^{-1}]^\top,\\
    a^m= J^mA^m= \frac{1}{(1+t)^{d-1}}{\rm det}\nabla(\widetilde{u} + (1+t)y)[(\nabla(\widetilde{u} + (1+t)y))^{-1}]^{\top}.\\
	\end{cases}
  \end{eqnarray}
Substituting  \eqref{local_existence_3_1_1} and \eqref{local_existence_3_1_2} into \eqref{local_existence_3_1}, we obtain
\begin{equation*}\label{local_existence_3_2}
  \begin{split}
    \partial_t^2 \widetilde{u}_{i} -\Delta \widetilde{u}_{i} = \,&\gamma \Big({\rm det}\nabla(\widetilde{u} + (1+t)y)\Big)^{2-(\gamma +1)}[\nabla(\widetilde{u} + (1+t)y)]^{-\top}_{il}\,\\ & [\nabla(\widetilde{u} + (1+t)y)]^{-\top}_{js} \partial_s\partial_l \widetilde{u}_{j}.
  \end{split}
\end{equation*}
This can be simplified to
\begin{equation*}\label{local_existence_3_3}
  \begin{split}
    &\partial_t^2 \widetilde{u}_{i} -\Delta \widetilde{u}_{i} - \gamma \big({\rm det}[(1+t)I]\big)^{1-\gamma} [(1+t)I]^{-\top}_{il}\, [(1+t)I]^{-\top}_{js} \partial_s\partial_l \widetilde{u}_{j}\\ &= \gamma \Big[\big({\rm det}[\nabla(\widetilde{u} + (1+t)y)]\big)^{1-\gamma}[\nabla(\widetilde{u} + (1+t)y)]^{-\top}_{il} [\nabla(\widetilde{u} + (1+t)y)]^{-\top}_{js} \\ &- \big({\rm det}[(1+t)I]\big)^{1-\gamma} [(1+t)I]^{-\top}_{il}\, [(1+t)I]^{-\top}_{js}\Big]\partial_s\partial_l \widetilde{u}_{j}.
  \end{split}
\end{equation*}
A further simplification yields
\begin{equation}\label{local_existence_3_4}
  \begin{split}
    \partial_t^2 \widetilde{u} -\Delta \widetilde{u} - \gamma (1+t)^{d(1-\gamma)-2}\nabla {\rm div} \widetilde{u}=F(\nabla\widetilde{u}, \nabla^2\widetilde{u}),
  \end{split}
\end{equation}
where $F(\nabla\widetilde{u}, \nabla^2\widetilde{u})$ is linear with respect to $\nabla^2\widetilde{u}$.
The almost global existence of classical solutions for \eqref{local_existence_3_4} follows from the framework in \cite{Klainerman_1996}, with a detailed treatment given in \cite[Section 15.3.1]{Li_2013}. Therefore, we can construct a local classical solution $(v^m, \widetilde{\eta}^m)$ to \eqref{local_existence_3} on $[0, \widetilde{T}^m]$ for some $\widetilde{T}^m>0$.
  \subsubsection*{Step 2: Uniform estimates}

By continuity, for any fixed $m>0,$ there exists $T^m\in(0,\widetilde{T}^m)$ such that for $0\le t\le T^m,$
\begin{equation}\label{local_existence_5}
  \begin{split}
   \frac{1}{4M_2}\le{\rm det}[\nabla \widetilde{\eta}^m(t)+I]\le 4M_2.
  \end{split}
\end{equation}
Furthermore, substituting  \eqref{perturbation_1_1} and \eqref{perturbation_2} into \eqref{flow_map_2} and \eqref{flow_map_5_1}-\eqref{flow_map_7}, we derive the following geometric identities:
 \begin{equation}\label{local_existence_7}
  \begin{split}
 \nabla A_{ij}^m= -A_{il}^m\partial_l\nabla\widetilde{\eta}_{k}^mA_{kj}^m,\,\,\,\,\,\,\,\,\,\,\,\,\,\nabla a_{ij}^m= a_{lk}^m\partial_k\nabla\widetilde{\eta}_{l}^mA_{ij}^m- a_{il}^m\partial_l\nabla\widetilde{\eta}_{k}^mA_{kj}^m,\,\,\,\,\,\,\,\,\,\,\,\,\,\,\partial_l a_{il}^m=0.
\end{split}
\end{equation}
 Using \eqref{local_existence_7_1}, \eqref{local_existence_5} and \eqref{local_existence_7} and $q^m= (J^m)^{-\gamma}$, we derive
 \begin{eqnarray}\label{local_existence_9_1}
   \begin{cases}
 |a^m, J^m, A^m|\le C,\\
 |\nabla (a^m, J^m, q^m)| \le C |\nabla^2 \widetilde{\eta}^m|,\\
 |\nabla^2 (a^m, J^m, q^m)| \le C |\nabla^3 \widetilde{\eta}^m| + C|\nabla^2 \widetilde{\eta}^m|^2,\\
 |\nabla^3 (a^m, J^m, q^m)|\le C|\nabla^4 \widetilde{\eta}^m| +C|\nabla^2 \widetilde{\eta}^m||\nabla^3 \widetilde{\eta}^m| +C|\nabla^2 \widetilde{\eta}^m|^3,\\
  |\partial_t q^m|\le C| \nabla v^m|,\\
  |\nabla \partial_t q^m|\le C| \nabla^2 v^m| +C|\nabla v^m||\nabla^2 \widetilde{\eta}^m|,\\
  |\nabla^2 \partial_t q^m|\le C|\nabla^3 v^m|+C|\nabla^2 v^m||\nabla^2 \widetilde{\eta}^m| +C|\nabla v^m||\nabla^3 \widetilde{\eta}^m| +C|\nabla v^m||\nabla^2 \widetilde{\eta}^m||\nabla^2 \widetilde{\eta}^m|.
 \end{cases}
  \end{eqnarray}
Applying $\nabla^s (s=0,1,2,3)$ to \eqref{local_existence_3}$_2$, multiplying the result by $\nabla^s v^m_i$, respectively, summing
over $0\le s\le 3$, and integrating the resulting equations over $\mathbb{R}^d$ and applying integration by parts, we obtain
\begin{equation}\label{local_existence_6}
   \begin{split}
    &\frac{1}{2}\udt (\|v^m\|_{H^3}^2 + \|\nabla \widetilde{\eta}^m\|_{H^3}^2) + \frac{2}{1+t}\|v^m\|_{H^3}^2 \\ &+ \frac{1}{(1+t)^{d(\gamma -1)+2}}\sum_{s=0}^3\int_{\mathbb{R}^d} \nabla^s(a_{il}^m\partial_l q^m) \cdot\nabla^s v_i^m\,dy =0.
   \end{split}
\end{equation}
In \eqref{local_existence_6}, we have
\begin{equation}\label{local_existence_8}
  \begin{split}
      &\frac{1}{(1+t)^{d(\gamma -1)+2}}\sum_{s = 0}^{3}\int_{\mathbb{R}^d} \nabla^s(a_{il}^m\partial_l q^m) \cdot\nabla^s v_i^m\,dy\\ & = \frac{1}{(1+t)^{d(\gamma -1)+2}}\int_{\mathbb{R}^d} a_{il}^m\nabla^3\partial_l q^m\cdot\nabla^3 v_i^m\,dy + R_1.
  \end{split}
\end{equation}
The term $R_1$ can be estimated by
\begin{equation}\label{local_existence_9}
  \begin{split}
    |R_1|\le &\,C \|\nabla^3 v^m\|_{L^2}(\|\nabla q^m\|_{L^\infty}\|\nabla^3 a^m\|_{L^2} + \|\nabla^2 q^m\|_{L^4}\|\nabla^2 a^m\|_{L^4} + \|\nabla a^m\|_{L^\infty}\|\nabla^3 q^m\|_{L^2})\\ & + C \|\nabla^2 v^m\|_{L^2}(\|\nabla^3 q^m\|_{L^2} + \|\nabla a^m\|_{L^\infty}\|\nabla^2 q^m\|_{L^2} + \|\nabla q^m\|_{L^\infty}\|\nabla^2 a^m\|_{L^2})\\ & + C \|\nabla v^m\|_{L^2}(\|\nabla^2 q^m\|_{L^2} +\|\nabla a^m\|_{L^\infty}\|\nabla q^m\|_{L^2}) + C \|v^m\|_{L^2}\|\nabla q^m\|_{L^2}\\\le &\,C \|v^m\|_{H^3}\|\nabla q^m\|_{H^2}(1+\|\nabla a^m\|_{H^2})\\\le &\,C \|v^m\|_{H^3}^2 + C\|\nabla q^m\|_{H^2}^2 + C\|\nabla q^m\|_{H^2}^4 + C\|\nabla a^m\|_{H^2}^4\\ \le &\,C \|v^m\|_{H^3}^2 + C\|\widetilde{\eta}^m\|_{H^4}^{12} +C,
  \end{split}
\end{equation}
where we have used the Cauchy inequality, \eqref{local_existence_9_1}, and the Sobolev inequality
\begin{equation}\label{sobolev}
	\begin{split}
		\|f\|_{L^\infty}\le\,  C \|f\|_{H^2},\,\,\,\,\,\,\,\,\,\,\,\,\,\,\,\,\,\|f\|_{L^4}\le\,  C \|f\|_{H^1}.
	\end{split}
\end{equation}
Furthermore, using \eqref{local_existence_7_1} and $q^m= (J^m)^{-\gamma}$,  we deduce
\begin{equation}\label{local_existence_10}
  \begin{split}
    \partial_t J^m= a_{ij}^m\partial_j v_i^m= -\frac{(J^m)^{\gamma +1}}{\gamma} \partial_t q^m.
  \end{split}
\end{equation}
Applying \eqref{local_existence_10} in \eqref{local_existence_8} and using integration by parts, we obtain
\begin{equation}\label{local_existence_11}
  \begin{split}
    &\frac{1}{(1+t)^{d(\gamma -1)+2}}\int_{\mathbb{R}^d} a_{il}^m\nabla^3\partial_l q^m\cdot\nabla^3 v_i^m\,dy\\ & = - \frac{1}{(1+t)^{d(\gamma -1)+2}}\int_{\mathbb{R}^d} a_{il}^m\nabla^3 q^m\cdot\partial_l\nabla^3 v_i^m\,dy\\ &  = -\frac{1}{(1+t)^{d(\gamma -1)+2}}\Big(\int_{\mathbb{R}^d} \nabla^3 q^m\cdot\nabla^3\partial_t J^m\,dy- \int_{\mathbb{R}^d} \nabla^3 q^m\cdot\nabla^3 a_{il}^m\partial_l v_i^m\,dy  \\ &- 3\int_{\mathbb{R}^d} \nabla^3 q^m\cdot\big(\nabla^2 a_{il}^m\nabla\partial_l v_i^m\big) \,dy -  3\int_{\mathbb{R}^d} \nabla^3 q^m\cdot\big(\nabla a_{il}^m\nabla^2\partial_l v_i^m\big) \,dy\Big) \\ & =: \,R_2 + R_3 + R_4 + R_5.
  \end{split}
\end{equation}
In \eqref{local_existence_11}, using \eqref{local_existence_10}, we obtain
\begin{equation}\label{local_existence_12}
  \begin{split}
    R_2 = &\frac{1}{\gamma (1+t)^{d(\gamma -1)+2}}\int_{\mathbb{R}^d}\nabla^3 q^m\cdot\nabla^3\big[(J^m)^{\gamma +1}\partial_t q^m\big] \,dy\\  = &\frac{1}{\gamma (1+t)^{d(\gamma -1)+2}}\Big[\int_{\mathbb{R}^d} (J^m)^{\gamma +1}\nabla^3 q^m\cdot \nabla^3\partial_t q^m\,dy + \int_{\mathbb{R}^d} \nabla^3 q^m\cdot\nabla^3(J^m)^{\gamma +1}\partial_t q^m\,dy\\ & + 3\int_{\mathbb{R}^d} \nabla^3 q^m\cdot\big(\nabla^2(J^m)^{\gamma +1}\nabla\partial_t q^m\big) \,dy+ 3\int_{\mathbb{R}^d} \nabla^3 q^m\cdot\big(\nabla (J^m)^{\gamma +1}\nabla^2\partial_tq^m\big) \,dy\Big] \\ =&:R_6 + R_7 + R_8 + R_9.
  \end{split}
\end{equation}
For $R_6$ in \eqref{local_existence_12}, we have
\begin{equation}\label{local_existence_13}
  \begin{split}
    R_6 = &\,\frac{1}{2\gamma (1+t)^{d(\gamma -1)+2}}\udt\int_{\mathbb{R}^d} (J^m)^{\gamma +1}|\nabla^3 q^m|^2 \,dy - \frac{(\gamma +1)}{2\gamma (1+t)^{d(\gamma -1)+2}}\int_{\mathbb{R}^d} (J^m)^{\gamma}\\ &\partial_t J^m|\nabla^3 q^m|^2 \,dy\\  = &\,\udt\Big[\frac{1}{2\gamma (1+t)^{d(\gamma -1)+2}}\int_{\mathbb{R}^d} (J^m)^{\gamma +1}|\nabla^3 q^m|^2 \,dy\Big] + \frac{d(\gamma -1)+2}{2\gamma (1+t)^{d(\gamma -1)+3}}\\ &\int_{\mathbb{R}^d} (J^m)^{\gamma +1}|\nabla^3 q^m|^2 dy- \frac{(\gamma +1)}{2\gamma (1+t)^{d(\gamma -1)+2}}\int_{\mathbb{R}^d} (J^m)^{\gamma}\partial_t J^m|\nabla^3 q^m|^2 \,dy \\ =&: R_{10} + R_{11} + R_{12}.
  \end{split}
\end{equation}
 Using \eqref{local_existence_7_1}, \eqref{local_existence_7}-\eqref{local_existence_9_1}, \eqref{sobolev} and the Cauchy inequality, the terms $R_3, R_4, R_5$, $R_7, R_8, R_9$ and $R_{12}$ from \eqref{local_existence_11}-\eqref{local_existence_13} can be estimated as
\begin{equation}\label{local_existence_14}
  \begin{split}
    &|R_3| + |R_4| + |R_5| + |R_7| + |R_8| + |R_9| + |R_{12}|\\ &\le C \|v^m\|_{H^3}\|\nabla q^m\|_{H^2}\|\nabla a^m\|_{H^2} + C\|\nabla q^m\|_{H^2}\|\partial_t q^m\|_{H^2}(\|\nabla^3 J^m\|_{L^2} \\ &+ \|\nabla^2 J^m\|_{L^4}\|\nabla J^m\|_{L^\infty} +\|\nabla J^m\|_{L^\infty}^3 +1)  + C\|\nabla q^m\|_{H^2}^2\|\nabla v^m\|_{H^2}\\ & \le C\|v^m\|_{H^3}^2 +C\|\nabla q^m\|_{H^2}^4 + C\|\nabla a^m\|_{H^2}^4 + C\|\nabla q^m\|_{H^2}^2 + C\|\partial_t q^m\|_{H^2}^2 + C\|\partial_t q^m\|_{H^2}^4 \\ &+ C\|\nabla^3 J^m\|_{L^2}^4 + C\|\nabla^2 J^m\|_{L^4}^4\|\nabla J^m\|_{L^\infty}^4 + C\|\nabla J^m\|_{L^\infty}^{12}\\ &\le C \|v^m\|_{H^3}^8 + C\|\widetilde{\eta}^m\|_{H^4}^{16} +C.
  \end{split}
\end{equation}
Substituting \eqref{local_existence_8}, \eqref{local_existence_9}, \eqref{local_existence_11}-\eqref{local_existence_14} into \eqref{local_existence_6}, we obtain
\begin{equation}\label{local_existence_15}
  \begin{split}
    &\frac{1}{2}\udt \big[\|v^m\|_{H^3}^2 + \|\nabla \widetilde{\eta}^m\|_{H^3}^2 + \frac{1}{\gamma (1+t)^{d(\gamma -1)+2}} \int_{\mathbb{R}^d}(J^m)^{\gamma +1}|\nabla^3 q^m|^2 \,dy\big]\\ &\le C \|v^m\|_{H^3}^8 + C\|\widetilde{\eta}^m\|_{H^4}^{16} +C.
  \end{split}
\end{equation}
Multiplying \eqref{local_existence_3}$_1$ by $\widetilde{\eta}_i^m$ respectively,  
integrating the resulting equations and applying the {H}\"{o}lder inequality, we have
\begin{equation}\label{local_existence_16}
   \begin{split}
    \frac{1}{2}\udt \|\widetilde{\eta}^m\|_{L^2}^2 = \langle v_i^m, \widetilde{\eta}_i^m\rangle\le C\|v^m\|_{L^2}\|\widetilde{\eta}^m\|_{L^2}.
   \end{split}
\end{equation}
Combining \eqref{local_existence_15} with \eqref{local_existence_16}, we derive
\begin{equation}\label{local_existence_17}
  \begin{split}
    &\frac{1}{2}\udt \big[\|v^m\|_{H^3}^2 + \|\widetilde{\eta}^m\|_{H^4}^2 + \frac{1}{\gamma (1+t)^{d(\gamma -1)+2}} \int_{\mathbb{R}^d}(J^m)^{\gamma +1}|\nabla^3 q^m|^2 \,dy\big]\\ &\le C \|v^m\|_{H^3}^8 + C\|\widetilde{\eta}^m\|_{H^4}^{16} +C.
  \end{split}
\end{equation}
Integrating \eqref{local_existence_17} over $[0, t]$ for any $t\in[0,T^m]$, then taking the supremum over $t\in[0,T^m]$, we deduce
   \begin{equation*}\label{local_estimate_18}
	   \begin{split}
		&\sup_{0\leq t \leq T^m}\big(\|v^m(t)\|_{H^3}^2 + \|\widetilde{\eta}^m(t)\|_{H^4}^2\big) \le CM_1 + C T^m\sup_{0\leq t \leq T^m}(\|v^m(t)\|_{H^3}^8+\|\widetilde{\eta}^m(t)\|_{H^4}^{16} + 1).
	   \end{split}
   \end{equation*}
As in \cite[Section 9]{Coutand_2006}, the above inequalities provide a time $T_1\le T^m$ independent of $m$ such that
\begin{equation*}\label{local_estimate_20}
  \begin{split}
     \sup_{0\leq t \leq T_1}\big(\|v^m(t)\|_{H^3}^2 + \|\widetilde{\eta}^m(t)\|_{H^4}^2\big) \le CM_1.
  \end{split}
\end{equation*}
\subsubsection*{Step 3: Local existence and uniqueness}
By standard weak (or weak*) convergence results and the Aubin-Lions Lemma (see, for instance, \cite{Simon_1987}), there exists a limit $(v, \widetilde{\eta})$ in $C([0, T_1]; H^3(\mathbb{R}^d)\times H^4(\mathbb{R}^d))$ of $(v^m, \widetilde{\eta}^m)$ (taking a subsequence if necessary) that solves the Cauchy problem \eqref{compressible_elasticity_equation_5}. In fact, the continuity in time of the limit can be achieved in a way similar to that in Section 6.3 of \cite{Li_2013}, and the limit satisfies
\begin{equation}\label{local_estimate_25}
  \begin{split}
     \sup_{0\leq t \leq T_1}\big(\|v(t)\|_{H^3}^2 + \|\widetilde{\eta}(t)\|_{H^4}^2\big) \le CM_1.
  \end{split}
\end{equation}
For uniqueness, suppose there are two solutions $(v, \widetilde{\eta})$ and $(w, \phi)$ to the system \eqref{compressible_elasticity_equation_5}. Let $\overline{v}: = v - w,$ $\overline{\eta}: = \widetilde{\eta} - \phi$, which satisfy
\begin{eqnarray} \label{local_estimate_21}
	\begin{cases}
		\partial_t \overline{\eta}=\overline{v},\\
		\partial_t \overline{v}_{i} + \frac{2}{1+t} \overline{v}_{i} + \frac{1}{(1+t)^{d(\gamma -1)+2}} \Big(a_{il}\partial_l q(\widetilde{\eta})- a_{il}\partial_l q(\phi)\Big) -\Delta \overline{\eta}_{i} =0.
	\end{cases}
  \end{eqnarray}
Multiplying \eqref{local_estimate_21}$_1$ and \eqref{local_estimate_21}$_2$ by $(\overline{\eta}, \overline{v}_{i})$, respectively, 
and integrating the resulting equations over $\mathbb{R}^d$, we obtain that for all $0\le t\le T_1,$
\begin{equation}\label{local_estimate_22}
  \begin{split}
    &\frac{1}{2}\udt \Big(\|\overline{\eta}\|_{H^1}^2 + \|\overline{v}\|_{L^2}^2 \Big) - \frac{\gamma}{(1+t)^{d(\gamma -1)+2}}\int_{\mathbb{R}^d}\Big[\frac{A_{il}A_{sj}}{J^{\gamma - 1}}(\widetilde{\eta})\partial_j\partial_l\widetilde{\eta}_s \\ &- \frac{A_{il}A_{sj}}{J^{\gamma - 1}}(\phi)\partial_j\partial_l\phi_s\Big] \overline{v}_{i}\,dy \le \, C\Big(\|\overline{\eta}\|_{H^1}^2 + \|\overline{v}\|_{L^2}^2\Big).
  \end{split}
\end{equation}
For the second term on the left-hand side of the inequality, using integration by parts, we obtain
\begin{equation}\label{local_estimate_26}
  \begin{split}
 & - \frac{\gamma}{(1+t)^{d(\gamma -1)+2}}\int_{\mathbb{R}^d}\Big[\frac{A_{il}A_{sj}}{J^{\gamma - 1}}(\widetilde{\eta})\partial_j\partial_l\widetilde{\eta}_s- \frac{A_{il}A_{sj}}{J^{\gamma - 1}}(\phi)\partial_j\partial_l\phi_s\Big] \overline{v}_{i}\,dy\\
   &=  - \frac{\gamma}{(1+t)^{d(\gamma -1)+2}}\int_{\mathbb{R}^d}\frac{A_{il}A_{sj}}{J^{\gamma - 1}}(\widetilde{\eta})\partial_j\partial_l\overline{\eta}_s \overline{v}_{i} + \Big(\frac{A_{il}A_{sj}}{J^{\gamma - 1}}(\widetilde{\eta}) - \frac{A_{il}A_{sj}}{J^{\gamma - 1}}(\phi)\Big)\partial_j\partial_l\phi_s \overline{v}_{i}\,dy\\
    & = \frac{\gamma}{(1+t)^{d(\gamma -1)+2}}\Big[\int_{\mathbb{R}^d}\frac{A_{il}(\widetilde{\eta})\partial_l\overline{v}_{i} A_{sj}(\widetilde{\eta})\partial_j\overline{\eta}_s}{J^{\gamma - 1}(\widetilde{\eta})} \,dy  +\int_{\mathbb{R}^d}\partial_l\Big(\frac{A_{il}(\widetilde{\eta})A_{sj}(\widetilde{\eta})}{J^{\gamma - 1}(\widetilde{\eta})}\Big) \overline{v}_{i}\partial_j\overline{\eta}_s\,dy \\ 
     &- \Big(\frac{A_{il}A_{sj}}{J^{\gamma - 1}}(\widetilde{\eta}) - \frac{A_{il}A_{sj}}{J^{\gamma - 1}}(\phi)\Big)\partial_j\partial_l\phi_s \overline{v}_{i}\,dy\Big]\\ & =: R_{13} + R_{14} + R_{15}.
  \end{split}
\end{equation}
As for $R_{13}$ in \eqref{local_estimate_26}, we deduce
\begin{equation}\label{local_estimate_27}
  \begin{split}
    R_{13} = &\frac{\gamma}{(1+t)^{d(\gamma -1)+2}}\int_{\mathbb{R}^d}\frac{A_{il}(\widetilde{\eta})\partial_l\overline{v}_{i} A_{sj}(\widetilde{\eta})\partial_j\overline{\eta}_s}{J^{\gamma - 1}(\widetilde{\eta})} \,dy \\
    =& \frac{\gamma}{(1+t)^{d(\gamma -1)+2}}\int_{\mathbb{R}^d}\frac{\partial_t\big(A_{il}(\widetilde{\eta})\partial_l\overline{\eta}_{i}\big) A_{sj}(\widetilde{\eta})\partial_j\overline{\eta}_s}{J^{\gamma - 1}(\widetilde{\eta})}  - \frac{\partial_tA_{il}(\widetilde{\eta})\partial_l\overline{\eta}_{i} A_{sj}(\widetilde{\eta})\partial_j\overline{\eta}_s}{J^{\gamma - 1}(\widetilde{\eta})}\,dy\\
    =&
    \frac{1}{2} \udt\Big[\frac{\gamma}{(1+t)^{d(\gamma -1)+2}}\int_{\mathbb{R}^d}\frac{|A_{jk}(\widetilde{\eta})\partial_k \overline{\eta}_i|^2}{J^{\gamma - 1}(\widetilde{\eta})}\,dy\Big] \\
    +& \frac{\gamma(d(\gamma -1)+2)}{2(1+t)^{d(\gamma -1)+3}}\int_{\mathbb{R}^d}\frac{|A_{jk}(\widetilde{\eta})\partial_k \overline{\eta}_i|^2}{J^{\gamma-1}(\widetilde{\eta})}\,dy\\
   - &\frac{\gamma}{2(1+t)^{d(\gamma -1)+2}}\Big[\int_{\mathbb{R}^d}\partial_t\big(\frac{1}{J^{\gamma - 1}(\widetilde{\eta})}\big)A_{il}(\widetilde{\eta})\partial_l\overline{\eta}_{i} A_{sj}(\widetilde{\eta})\partial_j\overline{\eta}_s + \frac{\partial_tA_{il}(\widetilde{\eta})\partial_l\overline{\eta}_{i} A_{sj}(\widetilde{\eta})\partial_j\overline{\eta}_s}{J^{\gamma - 1}(\widetilde{\eta})}\,dy\Big] \\
   =:& R_{16} + R_{17} + R_{18} + R_{19}.
  \end{split}
\end{equation}
By applying the Sobolev inequalities in \eqref{sobolev} and \eqref{local_estimate_25}, we obtain the following estimates for the terms  
$R_{14}$, $R_{15}$, $R_{18}$, and $R_{19}$:
\begin{equation}\label{local_estimate_28}
  \begin{split}
|R_{14}| + |R_{15}| + |R_{18}| +|R_{19}|\le C(M_1)(\|\overline{\eta}\|_{H^1}^2 + \|\overline{v}\|_{L^2}^2).
\end{split}
\end{equation}
Combining \eqref{local_estimate_22} with \eqref{local_estimate_26}-\eqref{local_estimate_28}, we deduce
\begin{equation}\label{local_estimate_23}
  \begin{split}
    &\udt\Big(\|\overline{\eta}\|_{H^1}^2 + \|\overline{v}\|_{L^2}^2 + \frac{\gamma}{(1+t)^{d(\gamma -1)+2}}\int_{\mathbb{R}^d}\frac{|A_{jk}(\widetilde{\eta})\partial_k \overline{\eta}_i|^2}{J^{\gamma - 1}(\widetilde{\eta})}\,dy\Big)\\ 
    &\le C(M_1)(\|\overline{\eta}\|_{H^1}^2 + \|\overline{v}\|_{L^2}^2).
  \end{split}
\end{equation}
By applying Gronwall's inequality to \eqref{local_estimate_23}, we obtain
\begin{equation}\label{local_estimate_24}
  \begin{split}
    \|\overline{\eta}(t)\|_{H^1}^2 + \|\overline{v}(t)\|_{L^2}^2 \le C(M_1)(\|\overline{\eta}(0)\|_{H^1}^2 + \|\overline{v}(0)\|_{L^2}^2)= 0.
  \end{split}
\end{equation}
Thus, $\overline{v} = v - w =0$ and $\overline{\eta} = \widetilde{\eta} - \phi=0,$ which completes the proof of Proposition \ref{local_existence_1}.
\end{proof}
\subsection{A priori estimates}\label{subsection_4}
In this subsection, we will establish apriori estimates for the strong solutions to the system \eqref{compressible_elasticity_equation_5}. We assume that $T^{\ast}> 0$ is the maximum existence time for the solution to the Cauchy problem \eqref{compressible_elasticity_equation_5}. To begin with, we define
 \begin{equation}\label{L}
	\begin{split}
    \mathcal{L}(t) := &\|v(t)\|_{H^3}^2 + \frac{2}{(1+t)^2}\|\widetilde{\eta}(t)\|_{H^3}^2 + \|\nabla \widetilde{\eta}(t)\|_{H^3}^2+ \frac{2}{1+t}\sum_{i = 0}^{3} \langle\nabla^{i} v(t),\nabla^{i} \widetilde{\eta}(t)\rangle \\ &+ \frac{1}{\gamma(1+t)^{d(\gamma -1)+2}} \int_{\mathbb{R}^d}J^{\gamma+1}(t) |\nabla^3 q(t)|^2 \,dy.
	\end{split}
\end{equation}
Throughout the rest of this section, $C\ge 1$  denotes a generic positive constant depending on some known constants but independent of $\varepsilon_1, \delta, T^*$,  and $\bar{C_1}\ge C$ depends only on $C$. The main Proposition of this subsection is 
the following.
\begin{pro}\label{uniform_estimates}
 Under the conditions of Theorem \ref{global_existence_Lagrangian}, there exist sufficiently small positive constants $\varepsilon_1$, $\delta$, that are  independent of $T$,  such that if
    \begin{equation}\label{a_priori_hypothesis}
      \sup_{0\leq t \leq T}(1+t)\mathcal{L}(t) \leq \delta
    \end{equation}
    for any given $0<T<T^\ast$, then
    \begin{equation}\label{a_priori_hypothesis_1}
      \sup_{0\leq t \leq T}(1+t)\mathcal{L}(t) \leq \frac{\delta}{2},
    \end{equation}
  provided that
   \begin{equation*}\label{initial_data}
    \|v_0(y)\|_{H^3_y} + \|\widetilde{\eta}_0(y)\|_{H^4_y}\le \varepsilon_1,
   \end{equation*}
where $\varepsilon_1, \delta$ satisfy the following  conditions:
\begin{equation}\label{coefficients_conditions}
	\begin{split}
     {\delta = 4\bar{C_1} \varepsilon_1^2\le\frac{1}{2\bar{C_1}}.}
  \end{split}
\end{equation}
\end{pro}
The proof of Proposition \ref{uniform_estimates} relies on the following Lemmas \ref{lemma_regularity_0}-\ref{lemma_regularity_3}.
 \begin{lemma}\label{lemma_regularity_0}
	Under the assumptions of Proposition \ref{uniform_estimates}, there exists a positive constant $C$, that is independent of $T$, such that
	\begin{equation*}\label{est_L2_1}
	   \begin{split}
		    &\frac{1}{2}\udt \Big[\|v(t)\|_{H^2}^2 + \frac{2}{(1+t)^2}\|\widetilde{\eta}(t)\|_{H^2}^2 + \|\nabla \widetilde{\eta}(t)\|_{H^2}^2+ \frac{2}{1+t}\sum_{i = 0}^{2} \langle\nabla^i v(t),\nabla^i\widetilde{\eta}(t)\rangle \Big]\\  &+ \frac{1}{1+t}\|v(t)\|_{H^2}^2 + \frac{2}{(1+t)^3}\|\widetilde{\eta}(t)\|_{H^2}^2  + \frac{1}{1+t}\|\nabla \widetilde{\eta}(t)\|_{H^2}^2 + \frac{1}{(1+t)^2}\sum_{i = 0}^{2}\langle \nabla^i v(t),\nabla^i \widetilde{\eta}(t)\rangle\\ &\le \frac{C}{(1+t)^{2}}\big(\mathcal{L}(t) + \mathcal{L}^{\frac{3}{2}}(t)\big)
	   \end{split}
   \end{equation*}
   for all $0\leq t \leq T$.
\end{lemma}
\begin{proof}
  Combining \eqref{sobolev} with \eqref{a_priori_hypothesis} and using \eqref{perturbation_1_1}, we deduce
  \begin{eqnarray}\label{a_priori_hypothesis_2}
    \begin{cases}
      \|\nabla \eta(t) -I\|_{L^\infty}\le C\sqrt{\delta},\\[1mm]
      \|J -{\rm det}\nabla y\|_{L^\infty} = \|J -1\|_{L^\infty}\le C\sqrt{\delta},\\[1mm]
      \|A -[(\nabla y)^{-1}]^{\top}\|_{L^\infty} = \|A - I\|_{L^\infty}\le C\sqrt{\delta}
    \end{cases}
  \end{eqnarray}
  for all $0\le t\le T$.
The bounds \eqref{coefficients_conditions} and \eqref{a_priori_hypothesis_2} imply  that
\begin{equation}\label{a_priori_hypothesis_3}
  \frac{1}{2}\le J \le 2,\,\,\,\,\,\,\,\,\,\,\,\,\,\,\,\|A,a\|_{L^\infty}\le 4,\,\,\,\,\,\,\,\,\,\,\,\,\,\,\,\|\nabla \eta\|_{L^\infty}\le 2.
 \end{equation}
 Similar to \eqref{local_existence_7_1} and \eqref{local_existence_7}, we have the following geometric identities:
 \begin{equation}\label{geometric_identities}
  \begin{split}
 \nabla J = a_{ij}\partial_j\nabla \widetilde{\eta}_{i},\,\,\,\,\,\,\,\,\,\,\,\,\,\,\,\,\,\,\,\,\,\,\,\,\,\,\,\,\,\,\,\,\,\,\,\,\,\,\,\,\,\,\,\,\,\,\,\,\,\,\,\,\,\,\,\,\,\,\,\,\nabla A_{ij}= -A_{il}\partial_l\nabla\widetilde{\eta}_{k}A_{kj},\\\nabla a_{ij}= a_{lk}\partial_k\nabla\widetilde{\eta}_{l}A_{ij}- a_{il}\partial_l\nabla\widetilde{\eta}_{k}A_{kj},\,\,\,\,\,\,\,\,\,\,\,\,\,\,\,\,\,\,\,\,\,\,\,\,\,\,\,\,\,\,\,\,\,\,\,\,\,\,\,\,\,\,\,\,\,\,\partial_l a_{il} =0.
\end{split}
 \end{equation}
 Combining \eqref{perturbation_1_1} with \eqref{a_priori_hypothesis_3} and \eqref{geometric_identities}, we obtain the following relations for the gradients of $(a,J,q)$ in terms of the gradients of $\widetilde{\eta}$:
\begin{eqnarray}\label{geometric_identities_1}
  \begin{cases}
|\nabla (a,J,q)|\le C|\nabla^2 \widetilde{\eta}|,\\
|\nabla^2 (a,J,q)|\le C|\nabla^3 \widetilde{\eta}| +C| \nabla^2 \widetilde{\eta}|^2,\\
|\nabla^3 (a,J,q)|\le C|\nabla^4 \widetilde{\eta}| + C|\nabla^2 \widetilde{\eta}||\nabla^3 \widetilde{\eta}|+C|\nabla^2 \widetilde{\eta}|^3,\\
|\partial_t q|\le C|\nabla v|,\\
    |\nabla \partial_t q|\le C|\nabla^2 v| + C|\nabla v||\nabla^2 \widetilde{\eta}|,\\
    |\nabla^2 \partial_t q|\le C|\nabla^3 v|+C|\nabla^2 v||\nabla^2 \widetilde{\eta}| +C|\nabla v||\nabla^3 \widetilde{\eta}| +C|\nabla v||\nabla^2 \widetilde{\eta}|^2,
\end{cases}
 \end{eqnarray}
 where $q=J^{-\gamma}$ is defined as in \eqref{perturbation_1_1}.

Applying $\nabla^k(k=0,1,2)$ to (\ref{compressible_elasticity_equation_5})$_2$, multiplying the results by $\nabla^k v_i$, respectively, summing 
over $k=0,1,2$ and integrating the resulting equations over $\mathbb{R}^d$, and applying integration by parts, we obtain
  \begin{equation}\label{est_L2_2}
     \begin{split}
      &\frac{1}{2}\udt (\|v\|_{H^2}^2 + \|\nabla \widetilde{\eta}\|_{H^2}^2) + \frac{2}{1+t}\|v\|_{H^2}^2 \\ &+ \frac{1}{(1+t)^{d(\gamma -1)+2}}\sum_{k = 0}^{2}\int_{\mathbb{R}^d} \nabla^k(a_{il}\partial_l q) \nabla^k v_i \,dy =0.
     \end{split}
  \end{equation}
In \eqref{est_L2_2}, using integration by parts and \eqref{a_priori_hypothesis_3}-\eqref{geometric_identities_1}, we derive
\begin{equation}\label{est_L2_3}
  \begin{split}
    \frac{-1}{(1+t)^{d(\gamma -1)+2}}\int_{\mathbb{R}^d} a_{il}\partial_l q v_i \,dy\le \frac{C}{(1+t)^{d(\gamma -1)+2}}\|\nabla^2 \widetilde{\eta}\|_{L^2}\|v\|_{L^2},\,\,\,\,\,\,\,\,\,\,\,\,\,\,\,\,\,\,\,\,\,\,\,\,\,\,
  \end{split}
\end{equation}
\begin{equation}\label{est_H1_3}
      \begin{split}
        \frac{-1}{(1+t)^{d(\gamma -1)+2}}\int_{\mathbb{R}^d} \nabla(a_{il}\partial_l q)\cdot\nabla v_i \,dy = &\, \frac{1}{(1+t)^{d(\gamma -1)+2}}\int_{\mathbb{R}^d} a_{il}\partial_l q  \Delta v_i \,dy \\ \le \,& \frac{C}{(1+t)^{d(\gamma -1)+2}}\|\nabla^2 \widetilde{\eta}\|_{L^2} \|\nabla^2 v\|_{L^2},
      \end{split}
    \end{equation}
and
\begin{equation}\label{est_H2_3}
    \begin{split}
      &\frac{-1}{(1+t)^{d(\gamma -1)+2}}\int_{\mathbb{R}^d} \nabla^2(a_{il}\partial_l q) \cdot\nabla^2 v_i \,dy \\  &= \frac{1}{(1+t)^{d(\gamma -1)+2}}\int_{\mathbb{R}^d} \nabla(a_{il}\partial_l q) \cdot\nabla\Delta v_i \,dy\\ & = \frac{1}{(1+t)^{d(\gamma -1)+2}}
      \int_{\mathbb{R}^d} (\partial_l q \nabla a_{il}\cdot \nabla\Delta v_i + a_{il}\nabla\partial_l q \cdot\nabla\Delta v_i) \,dy\\  
      &\le  \frac{C}{(1+t)^{d(\gamma -1)+2}}\big(\|\nabla^2 \widetilde{\eta}\|_{L^\infty}\|\nabla^2 \widetilde{\eta}\|_{L^2}\|\nabla^3 v\|_{L^2} + \|\nabla^3 \widetilde{\eta}\|_{L^2}\|\nabla^3 v\|_{L^2}\big).
    \end{split}
  \end{equation}
Substituting \eqref{est_L2_3}- \eqref{est_H2_3} into \eqref{est_L2_2}, we obtain
\begin{equation}\label{est_L2_4}
  \begin{split}
    \frac{1}{2}\udt (\|v\|_{H^2}^2 + \|\nabla \widetilde{\eta}\|_{H^2}^2) + \frac{2}{1+t}\|v\|_{H^2}^2  \le &\,\frac{C}{(1+t)^{2}}\|\nabla \widetilde{\eta}\|_{H^2}^2 +\frac{C}{(1+t)^{2}}\|v\|_{H^3}^2\\ & +\frac{C}{(1+t)^{2}}\|\nabla^2 \widetilde{\eta}\|_{L^\infty}\|\nabla^2 \widetilde{\eta}\|_{L^2}\|\nabla^3 v\|_{L^2}.
  \end{split}
\end{equation}
Similarly, applying $\nabla^k(k=0,1,2)$ to (\ref{compressible_elasticity_equation_5})$_2$, multiplying the results by $ \frac{1}{1+t}\nabla^k\widetilde{\eta}_i$, respectively, summing over $k=0,1,2$ and integrating the resulting equations
over $\mathbb{R}^d$, and applying integration by parts, we get
  \begin{equation}\label{est_L2_5}
     \begin{split}
        &\frac{1}{1+t}\sum_{k = 0}^{2}\langle \nabla^k \partial_t v, \nabla^k\widetilde{\eta}\rangle + \frac{2}{(1+t)^2}\sum_{k = 0}^{2}\langle \nabla^k v, \nabla^k\widetilde{\eta}\rangle \\ &+ \frac{1}{(1+t)^{d(\gamma -1)+3}}\sum_{k = 0}^{2} \int_{\mathbb{R}^d}\nabla^k( a_{il}\partial_l q) \nabla^k\widetilde{\eta}_i \,dy + \frac{1}{1+t}\|\nabla \widetilde{\eta}\|_{H^2}^2=0.
     \end{split}
  \end{equation}
  In \eqref{est_L2_5}, using \eqref{a_priori_hypothesis_3}-\eqref{geometric_identities_1} and integration by parts, we deduce
  \begin{equation}\label{est_L2_6}
    \begin{split}
     &\frac{1}{1+t}\sum_{k = 0}^{2}\langle \nabla^k \partial_t v, \nabla^k\widetilde{\eta}\rangle\\ &= \frac{1}{1+t} \udt \sum_{k = 0}^{2}\langle\nabla^k v, \nabla^k\widetilde{\eta}\rangle - \frac{1}{1+t}\|v\|_{H^2}^2\\ &= \udt\big[\frac{1}{1+t}\sum_{k = 0}^{2}\langle \nabla^k v, \nabla^k\widetilde{\eta}\rangle
     \big] + \frac{1}{(1+t)^{2}}\sum_{k = 0}^{2}\langle \nabla^k v, \nabla^k \widetilde{\eta}\rangle - \frac{1}{1+t} \|v\|_{H^2}^2,\,\,\,\,\,\,\,\,\,\,\,\,\,\,\,\,\,\,\,\,\,\,
    \end{split}
 \end{equation}
 \begin{equation}\label{est_L2_7}
  \begin{split}
    \frac{-1}{(1+t)^{d(\gamma -1)+3}}\int_{\mathbb{R}^d} a_{il}\partial_l q \widetilde{\eta}_i \,dy\le \frac{C}{(1+t)^{d(\gamma -1)+3}}\|\nabla^2 \widetilde{\eta}\|_{L^2}\|\widetilde{\eta}\|_{L^2},\,\,\,\,\,\,\,\,\,\,\,\,\,\,\,\,\,\,\,\,\,\,\,\,\,\,\,\,\,\,\,\,\,
  \end{split}
\end{equation}
\begin{equation}\label{est_H1_7}
    \begin{split}
      \frac{-1}{(1+t)^{d(\gamma -1)+3}}\int_{\mathbb{R}^d} \nabla(a_{il}\partial_l q) \cdot\nabla \widetilde{\eta}_i \,dy  &= \frac{1}{(1+t)^{d(\gamma -1)+3}}\int_{\mathbb{R}^d} a_{il}\partial_l q \Delta \widetilde{\eta}_i \,dy\\ &\le \frac{C}{(1+t)^{d(\gamma -1)+3}}\|\nabla^2 \widetilde{\eta}\|_{L^2}^2\,\,\,\,\,\,\,\,\,\,\,\,\,\,\,\,\,\,\,\,\,\,\,\,\,\,\,\,\,\,\,\,\,
\end{split}
  \end{equation}
and
\begin{equation}\label{est_H2_7}
  \begin{split}
    \frac{-1}{(1+t)^{d(\gamma -1)+3}}\int_{\mathbb{R}^d} \nabla^2(a_{il}\partial_l q) \cdot\nabla^2 \widetilde{\eta}_i \,dy  &= \frac{-1}{(1+t)^{d(\gamma -1)+3}}\int_{\mathbb{R}^d} a_{il}\partial_l q \Delta\Delta \widetilde{\eta}_i \,dy\\ &\le \frac{C}{(1+t)^{d(\gamma -1)+3}}\|\nabla^2 \widetilde{\eta}\|_{L^2}\|\nabla^4 \widetilde{\eta}\|_{L^2}.
  \end{split}
\end{equation}
Substituting \eqref{est_L2_6}-\eqref{est_H2_7} into \eqref{est_L2_5}, we obtain
\begin{equation}\label{est_L2_8}
  \begin{split}
    &\udt\big[\frac{1}{1+t}\sum_{k = 0}^{2}\langle \nabla^k v, \nabla^k\widetilde{\eta}\rangle\big] + \frac{3}{(1+t)^{2}}\sum_{k = 0}^{2}\langle \nabla^k v, \nabla^k\widetilde{\eta}\rangle - \frac{1}{1+t} \|v\|_{H^2}^2 + \frac{1}{1+t}\|\nabla \widetilde{\eta}\|_{H^2}^2\\ &\le \frac{C}{(1+t)^{2}}\|\nabla^2 \widetilde{\eta}\|_{H^2}^2 + \frac{C}{(1+t)^{4}}\|\widetilde{\eta}\|_{L^2}^2.
  \end{split}
\end{equation}
Applying $\nabla^k(k=0,1,2)$ to (\ref{compressible_elasticity_equation_5})$_1$ , multiplying the results by $\widetilde{\eta}$, respectively, 
summing over $k=0,1,2$ and integrating the resulting equations over $\mathbb{R}^d$, we get
\begin{equation*}\label{est_L2_9}
  \begin{split}
    \frac{1}{2}\udt \|\widetilde{\eta}\|_{H^2}^2 = \sum_{k = 0}^{2}\langle \nabla^k v, \nabla^k\widetilde{\eta}\rangle,
  \end{split}
\end{equation*}
which  yields
%\begin{equation*}\label{est_L2_10}
%  \begin{split}
%    \frac{1}{2(1+t)^2} \udt \|\widetilde{\eta}\|_{H^2}^2 = \frac{1}{(1+t)^2} \sum_{k = 0}^{2}\langle \nabla^k v, \nabla^k\widetilde{\eta}\rangle
%  \end{split}
%\end{equation*}
%and
\begin{equation}\label{est_L2_11}
  \begin{split}
    \frac{1}{2} \udt\big[\frac{1}{(1+t)^2} \|\widetilde{\eta}\|_{H^2}^2\big] + \frac{1}{(1+t)^3} \|\widetilde{\eta}\|_{H^2}^2 = \frac{1}{(1+t)^2} \sum_{k = 0}^{2}\langle \nabla^k v, \nabla^k\widetilde{\eta}\rangle.
  \end{split}
\end{equation}
Summing \eqref{est_L2_4}, \eqref{est_L2_8} and $2\times$\eqref{est_L2_11}, we complete the proof of Lemma \ref{lemma_regularity_0}.
\end{proof}
To establish  \eqref{a_priori_hypothesis_1}, we proceed with the following highest-order derivative estimates.
\begin{lemma}\label{lemma_regularity_3}
  Under the assumptions of Proposition \ref{uniform_estimates}, there exists a positive constant $C$,
  that is independent of $T$, such that
  \begin{equation*}\label{est_H3_1}
     \begin{split}
        &\frac{1}{2}\udt \Big[\|\nabla^3 v(t)\|_{L^2}^2 + \frac{2}{(1+t)^2}\|\nabla^3\widetilde{\eta}(t)\|_{L^2}^2 + \|\nabla^4 \widetilde{\eta}(t)\|_{L^2}^2 + \frac{1}{\gamma(1+t)^{d(\gamma -1)+2}} \int_{\mathbb{R}^d}J^{\gamma+1}(t) \\  &|\nabla^3 q(t)|^2 \,dy + \frac{2}{1+t}\langle \nabla^3 v(t), \nabla^3\widetilde{\eta}(t)\rangle \Big] + \frac{1}{1+t}\|\nabla^3 v(t)\|_{L^2}^2 + \frac{2}{(1+t)^3}\|\nabla^3 \widetilde{\eta}(t)\|_{L^2}^2\\  &+ \frac{1}{1+t}\|\nabla^4 \widetilde{\eta}(t)\|_{L^2}^2  + \frac{d(\gamma -1)+2}{2\gamma(1+t)^{d(\gamma -1)+3}}\int_{\mathbb{R}^d} J^{\gamma +1}|\nabla^3 q|^2 \,dy + \frac{1}{(1+t)^2}\langle \nabla^3 v(t), \nabla^3\widetilde{\eta}(t)\rangle\\ 
        &\le\frac{C}{(1+t)^{2}}\big(\mathcal{L}(t) + \mathcal{L}^{\frac{7}{2}}(t)\big)
     \end{split}
   \end{equation*}
   for all $0\leq t \leq T$.
\end{lemma}
\begin{proof}
  Multiplying $\nabla^3$(\ref{compressible_elasticity_equation_5})$_2$ by $\nabla^3 v_i$, respectively, summing the resuls
  and integrating over $\mathbb{R}^d$, and applying integration by parts, we get
  \begin{equation}\label{est_H3_2}
     \begin{split}
      &\frac{1}{2}\udt (\|\nabla^3 v\|_{L^2}^2 + \|\nabla^4 \widetilde{\eta}\|_{L^2}^2) + \frac{2}{1+t}\|\nabla^3 v\|_{L^2}^2 \\ &+ \frac{1}{(1+t)^{d(\gamma -1)+2}}\int_{\mathbb{R}^d} \nabla^3(a_{il}\partial_l q) \cdot\nabla^3 v_i \,dy =0.
     \end{split}
  \end{equation}
In \eqref{est_H3_2}, we deduce
  \begin{equation}\label{est_H3_3}
    \begin{split}
      &\frac{1}{(1+t)^{d(\gamma -1)+2}}\int_{\mathbb{R}^d} \nabla^3(a_{il}\partial_l q) \cdot\nabla^3 v_i \,dy\\ & = \frac{1}{(1+t)^{d(\gamma -1)+2}}\Big(\int_{\mathbb{R}^d} \partial_l q\nabla^3 a_{il} \cdot\nabla^3 v_i \,dy + 3\int_{\mathbb{R}^d} (\nabla^2 a_{il}\nabla\partial_l q)\cdot \nabla^3 v_i \,dy \\ &+ 3\int_{\mathbb{R}^d} (\nabla a_{il}\nabla^2\partial_l q)\cdot \nabla^3 v_i \,dy + \int_{\mathbb{R}^d} a_{il}\nabla^3\partial_l q \cdot\nabla^3 v_i \,dy\Big)=: I_1 + I_2 + I_3 + I_4.
    \end{split}
  \end{equation}
Using \eqref{a_priori_hypothesis_3}-\eqref{geometric_identities_1}, the terms  $I_1, I_2,  I_3$ can be estimated by
\begin{equation}\label{est_H3_4}
  \begin{split}
    |I_1| + |I_2| + |I_3|\le&\, \frac{C}{(1+t)^{2}}\|\nabla^2 \widetilde{\eta}\|_{L^\infty}\|\nabla^3 v\|_{L^2}(\|\nabla^4 \widetilde{\eta}\|_{L^2} + \|\nabla^2 \widetilde{\eta}\|_{L^\infty}\|\nabla^3 \widetilde{\eta}\|_{L^2} \\ &+ \|\nabla^2 \widetilde{\eta}\|_{L^\infty}^2\|\nabla^2 \widetilde{\eta}\|_{L^2}) + \|\nabla^3 v\|_{L^2}(\|\nabla^3 \widetilde{\eta}\|_{L^4} + \|\nabla^2 \widetilde{\eta}\|_{L^\infty}\|\nabla^2 \widetilde{\eta}\|_{L^4})^2\\ \le&\, \frac{C}{(1+t)^{2}}\big(\mathcal{L}^{\frac{3}{2}}(t) + \mathcal{L}^2(t) + \mathcal{L}^{\frac{5}{2}}(t)\big).
  \end{split}
\end{equation}
As for $I_4$ in \eqref{est_H3_3}, using \eqref{local_existence_11}-\eqref{local_existence_14},
we obtain
\begin{equation}\label{est_H3_5}
  \begin{split}
    I_4 = &\frac{1}{2}\udt \Big[\frac{1}{\gamma(1+t)^{d(\gamma -1)+2}}\int_{\mathbb{R}^d} J^{\gamma +1}|\nabla^3 q|^2 dy\Big] + \frac{d(\gamma -1)+2}{2\gamma(1+t)^{d(\gamma -1)+3}}\int_{\mathbb{R}^d} J^{\gamma +1}|\nabla^3 q|^2 dy\\ 
    &- \frac{(\gamma +1)}{2\gamma(1+t)^{d(\gamma -1)+2}}\int_{\mathbb{R}^d} J^{\gamma}\partial_t J|\nabla^3 q|^2 \,dy \\
    &+ \frac{1}{\gamma(1+t)^{d(\gamma -1)+2}}\Big[\int_{\mathbb{R}^d}\nabla^3 q \cdot\nabla^3(J^{\gamma +1})\partial_t q \,dy + 3\int_{\mathbb{R}^d} \nabla^3 q\cdot\Big(\nabla^2(J^{\gamma +1})\nabla\partial_t q \Big) \,dy\\
    &\qquad\qquad\qquad\qquad\qquad+ 3\int_{\mathbb{R}^d} \nabla^3 q\cdot\Big(\nabla (J^{\gamma +1})\nabla^2\partial_tq\Big)dy\Big]\\
    &+\frac{1}{(1+t)^{d(\gamma -1)+2}}\Big[\int_{\mathbb{R}^d} \nabla^3 q\cdot\nabla^3 a_{il}\partial_l v_i \,dy
    + 3\int_{\mathbb{R}^d} \nabla^3 q\cdot\big(\nabla^2 a_{il} \nabla\partial_l v_i\big)\,dy\\
    &\qquad\qquad\qquad\qquad + 3\int_{\mathbb{R}^d} \nabla^3 q\cdot\big(\nabla a_{il}\nabla^2\partial_l v_i\big) \,dy\Big]\\
    &=:I_5 + I_6 + I_7 + I_8 + I_9.
  \end{split}
\end{equation}
Using \eqref{a_priori_hypothesis_3}-\eqref{geometric_identities_1}, the terms  $I_7, I_8,  I_9$ can be estimated by
\begin{equation}\label{est_H3_6}
  \begin{split}
    |I_7|\le \frac{C}{(1+t)^{2}}\|\nabla v\|_{L^\infty}(\|\nabla^4 \widetilde{\eta}\|_{L^2} + \|\nabla^2 \widetilde{\eta}\|_{L^\infty}\|\nabla^3 \widetilde{\eta}\|_{L^2} + \|\nabla^2 \widetilde{\eta}\|_{L^\infty}^2\|\nabla^2 \widetilde{\eta}\|_{L^2})^2,
  \end{split}
\end{equation}
\begin{equation}\label{est_H3_7}
  \begin{split}
    |I_8|\le &\frac{C}{(1+t)^{2}}\|\nabla v\|_{L^\infty}(\|\nabla^4 \widetilde{\eta}\|_{L^2} + \|\nabla^2 \widetilde{\eta}\|_{L^\infty}\|\nabla^3 \widetilde{\eta}\|_{L^2} + \|\nabla^2 \widetilde{\eta}\|_{L^\infty}^2\|\nabla^2 \widetilde{\eta}\|_{L^2})^2 \\ & + \frac{C}{(1+t)^{2}}(\|\nabla^4 \widetilde{\eta}\|_{L^2} + \|\nabla^2 \widetilde{\eta}\|_{L^\infty}\|\nabla^3 \widetilde{\eta}\|_{L^2} + \|\nabla^2 \widetilde{\eta}\|_{L^\infty}^2\|\nabla^2 \widetilde{\eta}\|_{L^2})(\|\nabla^3 \widetilde{\eta}\|_{L^4} \\ &+ \|\nabla^2 \widetilde{\eta}\|_{L^\infty}\|\nabla^2 \widetilde{\eta}\|_{L^4})(\|\nabla^2 v\|_{L^4} + \|\nabla v\|_{L^\infty}\|\nabla^2 \widetilde{\eta}\|_{L^4})+ \frac{C}{(1+t)^{2}}\|\nabla^2 \widetilde{\eta}\|_{L^\infty}\\ &(\|\nabla^4 \widetilde{\eta}\|_{L^2} + \|\nabla^2 \widetilde{\eta}\|_{L^\infty}\|\nabla^3 \widetilde{\eta}\|_{L^2} + \|\nabla^2 \widetilde{\eta}\|_{L^\infty}^2\|\nabla^2 \widetilde{\eta}\|_{L^2})(\|\nabla^3 v\|_{L^2} + \|\nabla^2 \widetilde{\eta}\|_{L^\infty}\\ &\|\nabla^2 v\|_{L^2} + \|\nabla v\|_{L^\infty}\|\nabla^3 \widetilde{\eta}\|_{L^2} + \|\nabla v\|_{L^\infty}\|\nabla^2 \widetilde{\eta}\|_{L^4}^2),
  \end{split}
\end{equation}
and
\begin{equation}\label{est_H3_8}
  \begin{split}
    |I_9|\le &\frac{C}{(1+t)^{2}}\|\nabla v\|_{L^\infty}(\|\nabla^4 \widetilde{\eta}\|_{L^2} + \|\nabla^2 \widetilde{\eta}\|_{L^\infty}\|\nabla^3 \widetilde{\eta}\|_{L^2} + \|\nabla^2 \widetilde{\eta}\|_{L^\infty}^2\|\nabla^2 \widetilde{\eta}\|_{L^2})^2 \\ & + \frac{C}{(1+t)^{2}}\|\nabla^2 v\|_{L^4}(\|\nabla^4 \widetilde{\eta}\|_{L^2} + \|\nabla^2 \widetilde{\eta}\|_{L^\infty}\|\nabla^3 \widetilde{\eta}\|_{L^2} + \|\nabla^2 \widetilde{\eta}\|_{L^\infty}^2\|\nabla^2 \widetilde{\eta}\|_{L^2})\\ &(\|\nabla^3 \widetilde{\eta}\|_{L^4} + \|\nabla^2 \widetilde{\eta}\|_{L^\infty}\|\nabla^2 \widetilde{\eta}\|_{L^4})+ \frac{C}{(1+t)^{2}}\|\nabla^2 \widetilde{\eta}\|_{L^\infty}\|\nabla^3 v\|_{L^2}(\|\nabla^4 \widetilde{\eta}\|_{L^2} \\ &+ \|\nabla^2 \widetilde{\eta}\|_{L^\infty}\|\nabla^3 \widetilde{\eta}\|_{L^2} + \|\nabla^2 \widetilde{\eta}\|_{L^\infty}^2\|\nabla^2 \widetilde{\eta}\|_{L^2}).
  \end{split}
\end{equation}
Combining the estimates from \eqref{est_H3_6} to \eqref{est_H3_8} and using the Sobolev inequality \eqref{sobolev}, we derive
\begin{equation}\label{est_H3_8_1}
  \begin{split}
    |I_7| + |I_8| + |I_9|\le &\frac{C}{(1+t)^{2}}\big(\mathcal{L}^{\frac{3}{2}}(t) + \mathcal{L}^2(t)+ \mathcal{L}^{\frac{5}{2}}(t)+ \mathcal{L}^3(t)+ \mathcal{L}^{\frac{7}{2}}(t)\big).
  \end{split}
\end{equation}
Substituting \eqref{est_H3_3}-\eqref{est_H3_5} and \eqref{est_H3_8_1} into \eqref{est_H3_2}, we obtain
\begin{equation}\label{est_H3_9}
  \begin{split}
    &\frac{1}{2}\udt \big[\|\nabla^3 v\|_{L^2}^2 + \|\nabla^4 \widetilde{\eta}\|_{L^2}^2 + \frac{1}{\gamma(1+t)^{d(\gamma -1)+2}}\int_{\mathbb{R}^d} J^{\gamma +1}|\nabla^3 q|^2 \,dy\big] + \frac{2}{1+t}\|\nabla^3 v\|_{L^2}^2 \\ &+ \frac{d(\gamma -1)+2}{2\gamma(1+t)^{d(\gamma -1)+3}}\int_{\mathbb{R}^d} J^{\gamma +1}|\nabla^3 q|^2 \,dy\\ &\le \frac{C}{(1+t)^{2}}\big(\mathcal{L}^{\frac{3}{2}}(t) + \mathcal{L}^2(t)+ \mathcal{L}^{\frac{5}{2}}(t)+ \mathcal{L}^3(t)+ \mathcal{L}^{\frac{7}{2}}(t)\big).
  \end{split}
\end{equation}
Similarly, multiplying $\nabla^3$(\ref{compressible_elasticity_equation_5})$_2$ by $\frac{1}{1+t}\nabla^3\widetilde{\eta}_i$, respectively, 
summing and integrating the resulting equations, and applying integration by parts, we deduce
  \begin{equation}\label{est_H3_10}
     \begin{split}
      \frac{1}{1+t}\|\nabla^4 \widetilde{\eta}\|_{L^2}^2 &+ \frac{1}{1+t}\langle \partial_t \nabla^3 v, \nabla^3\widetilde{\eta}\rangle + \frac{2}{(1+t)^2}\langle \nabla^3 v, \nabla^3\widetilde{\eta}\rangle \\ &+ \frac{1}{(1+t)^{d(\gamma -1)+3}}\int_{\mathbb{R}^d} \nabla^3(a_{il}\partial_l q) \cdot\nabla^3 \widetilde{\eta}_i \,dy =0.
     \end{split}
  \end{equation}
  In \eqref{est_H3_10}, using integration by parts, \eqref{a_priori_hypothesis_3}-\eqref{geometric_identities_1}, we obtain
  \begin{equation}\label{est_H3_11}
    \begin{split}
      \frac{1}{1+t}\langle \partial_t \nabla^3 v, \nabla^3\widetilde{\eta}\rangle &= \frac{1}{1+t} \udt \langle \nabla^3 v, \nabla^3\widetilde{\eta}\rangle - \frac{1}{1+t} \|\nabla^3 v\|_{L^2}^2\\ &=\udt\big[\frac{1}{1+t}\langle \nabla^3 v, \nabla^3\widetilde{\eta}\rangle\big] + \frac{1}{(1+t)^{2}}\langle \nabla^3 v, \nabla^3\widetilde{\eta}\rangle - \frac{1}{1+t} \|\nabla^3 v\|_{L^2}^2
    \end{split}
 \end{equation}
 and
 \begin{equation}\label{est_H3_12}
  \begin{split}
    &\frac{-1}{(1+t)^{d(\gamma -1)+3}}\int_{\mathbb{R}^d} \nabla^3(a_{il}\partial_l q) \cdot\nabla^3 \widetilde{\eta}_i \,dy\\
    & = \frac{1}{(1+t)^{d(\gamma -1)+3}}\int_{\mathbb{R}^d} \nabla^2(a_{il}\partial_l q) \cdot\nabla^2\Delta \widetilde{\eta}_i \,dy\\  &= \,\frac{1}{(1+t)^{d(\gamma -1)+3}}\Big(\int_{\mathbb{R}^d} (\nabla^2a_{il}\partial_l q) \cdot\nabla^2\Delta \widetilde{\eta}_i \,dy + \int_{\mathbb{R}^d} 2(\nabla a_{il}\nabla\partial_l q) \cdot\nabla^2\Delta \widetilde{\eta}_i \,dy \\ &+\int_{\mathbb{R}^d} (a_{il}\nabla^2\partial_l q) \cdot\nabla^2\Delta \widetilde{\eta}_i \,dy\Big)\\  &\le\frac{C}{(1+t)^{d(\gamma -1)+3}}\|\nabla^2 \widetilde{\eta}\|_{L^\infty}\|\nabla^4 \widetilde{\eta}\|_{L^2}(\|\nabla^3 \widetilde{\eta}\|_{L^2} + \|\nabla^2 \widetilde{\eta}\|_{L^\infty}\|\nabla^2 \widetilde{\eta}\|_{L^2}) \\ & + \frac{C}{(1+t)^{d(\gamma -1)+3}}\|\nabla^4 \widetilde{\eta}\|_{L^2}(\|\nabla^4 \widetilde{\eta}\|_{L^2} + \|\nabla^2 \widetilde{\eta}\|_{L^\infty}\|\nabla^3 \widetilde{\eta}\|_{L^2} + \|\nabla^2 \widetilde{\eta}\|_{L^\infty}^2\|\nabla^2 \widetilde{\eta}\|_{L^2})\\  &\le \frac{C}{(1+t)^{2}}(\mathcal{L}(t) +\mathcal{L}^{\frac{3}{2}}(t) + \mathcal{L}^2(t)).
  \end{split}
\end{equation}
Substituting \eqref{est_H3_11} and \eqref{est_H3_12} into \eqref{est_H3_10} and using the Sobolev inequalities \eqref{sobolev}, we deduce
\begin{equation}\label{est_H3_13}
  \begin{split}
    &\udt\big[\frac{1}{1+t}\langle \nabla^3 v, \nabla^3\widetilde{\eta}\rangle\big] + \frac{3}{(1+t)^{2}}\langle \nabla^3 v, \nabla^3\widetilde{\eta}\rangle - \frac{1}{1+t} \|\nabla^3 v\|_{L^2}^2 + \frac{1}{1+t}\|\nabla^4 \widetilde{\eta}\|_{L^2}^2\\ 
    &\le\frac{C}{(1+t)^{2}}(\mathcal{L}(t) +\mathcal{L}^{\frac{3}{2}}(t) + \mathcal{L}^2(t)).
  \end{split}
\end{equation}
Multiplying $\nabla^3$(\ref{compressible_elasticity_equation_5})$_1$ by $\nabla^3\widetilde{\eta}$, we obtain
\begin{equation*}\label{est_H3_14}
  \begin{split}
    \frac{1}{2}\udt \|\nabla^3\widetilde{\eta}\|_{L^2}^2 = \langle \nabla^3 v, \nabla^3\widetilde{\eta}\rangle,
  \end{split}
\end{equation*}
which implies
%\begin{equation*}\label{est_H3_15}
%  \begin{split}
%    \frac{1}{2(1+t)^2} \udt \|\nabla^3 \widetilde{\eta}\|_{L^2}^2 = \frac{1}{(1+t)^2} \langle \nabla^3 v, \nabla^3\widetilde{\eta}\rangle
%  \end{split}
%\end{equation*}
%and
\begin{equation}\label{est_H3_16}
  \begin{split}
    \frac{1}{2} \udt\big[\frac{1}{(1+t)^2} \|\nabla^3 \widetilde{\eta}\|_{L^2}^2\big] + \frac{1}{(1+t)^3} \|\nabla^3\widetilde{\eta}\|_{L^2}^2 = \frac{1}{(1+t)^2} \langle \nabla^3 v, \nabla^3\widetilde{\eta}\rangle.
  \end{split}
\end{equation}
Summing \eqref{est_H3_9}, \eqref{est_H3_13} and $2\times$\eqref{est_H3_16}, and applying Young's inequality, we complete the proof of Lemma \ref{lemma_regularity_3}.
\end{proof}
\subsection{Proof of Proposition \ref{uniform_estimates}}\label{subection_5}
Combining Lemma \ref{lemma_regularity_0} with Lemma \ref{lemma_regularity_3}, and using Young's inequality, we obtain
\begin{equation}\label{est_L_1}
  \begin{split}
  \udt \mathcal{L}(t) + \frac{1}{1+t} \mathcal{L}(t)\le \frac{C}{(1+t)^{2}}\big(\mathcal{L}(t) + \mathcal{L}^{\frac{7}{2}}(t)\big).
  \end{split}
\end{equation}
   Multiplying both sides of \eqref{est_L_1} by $(1+t)\, \mathrm {e}^{C(1+t)^{-1}}$, we derive
   \begin{equation}\label{est_L_3}
	   \begin{split}
			\udt [\mathcal{L}(t)(1+t)\,  \mathrm {e}^{C(1+t)^{-1}}]  \le \frac{C\mathrm {e}^{C}}{1+t}\mathcal{L}^{\frac{7}{2}}(t).
	   \end{split}
   \end{equation}
 Define
   \begin{equation*}\label{M_1}
	   \mathcal{M}(t): = \sup_{0\leq s \leq t}(1+s)\mathcal{L}(s).
   \end{equation*}
 Thus, for all $0\leq t \leq T$, we have
   $$\mathcal{L}(t) \le (1+t)^{-1}\mathcal{M}(t).$$	
Substituting this into \eqref{est_L_3}, we immediately obtain
   \begin{equation}\label{est_L_4}
	   \begin{split}
		\udt &[\mathcal{L}(t)(1+t)\, \mathrm {e}^{C(1+t)^{-1}}]  \le \,C\mathrm {e}^{C}(1+t)^{-\frac{9}{2}}\mathcal{M}^{\frac{7}{2}}(t).
	   \end{split}
   \end{equation}
  Integrating (\ref{est_L_4}) over $[0, t]$ for any $t\in[0,T]$, we get
   \begin{equation}\label{est_L_5}
	   \begin{split}
      (1+t)\mathcal{L}(t)  \le& \,\bar{C_1}\mathcal{L}(0) + \bar{C_1}\mathcal{M}^{\frac{7}{2}}(t),
	   \end{split}
   \end{equation}
   where $\bar{C_1}$ depends on $C$ only.

   Taking the supremum with respect to $t$ from $0$ to $T$ in \eqref{est_L_5} and using \eqref{a_priori_hypothesis} and $\delta \le 1$ in \eqref{coefficients_conditions}, we deduce
   \begin{equation*}\label{est_L_6}
	   \begin{split}
		   \mathcal{M}(T)  \le& \bar{C_1}\mathcal{L}(0) + \bar{C_1}\mathcal{M}^{\frac{7}{2}}(t)
\\ \le& \bar{C_1}\varepsilon_1^2 + \bar{C_1}\delta\mathcal{M}(T),
	   \end{split}
   \end{equation*}
which together  with $\delta = 4\bar{C_1} \varepsilon_1^2$ and $\bar{C_1}\delta\le\frac{1}{2}$  in \eqref{coefficients_conditions} implies
   \begin{equation*}\label{est_L_7}
	   \begin{split}
		   \mathcal{M}(T) \le  \frac{\delta}{2}.
	   \end{split}
   \end{equation*}
   This completes the proof of Proposition \ref{uniform_estimates}.  With the uniform regularity result established, we can now demonstrate that \eqref{a_priori_hypothesis} holds throughout the maximal lifespan of the solution to the Cauchy problem \eqref{compressible_elasticity_equation_5}.
\begin{corollary}\label{corollary_1}
  Under the assumptions of Theorem \ref{global_existence_Lagrangian}, the following inequality
  \begin{equation}\label{a priori-conclusion}
      (1+t)\mathcal{L}(t) \le \delta
    \end{equation}
    holds for all $t\in[0,T^*)$, where $T^*$ is the maximal lifespan of the solution to the Cauchy problem \eqref{compressible_elasticity_equation_5}.
  \end{corollary}
  \begin{proof}
  Since
  \begin{equation*}
    (1+0)\mathcal{L}(0) \le C\varepsilon_1^2\le\frac{\delta}{4},
  \end{equation*}
 we conclude from the continuity-in-time of the solution that there exists a $T_2>0$ such that
   (\ref{a priori-conclusion}) holds for all $t\in[0,T_2]$. Let $T^{\ast\ast}\in(T_2,T^\ast]$ be the maximal time 
   such that (\ref{a priori-conclusion}) holds for all $t\in[0,T^{\ast\ast})$   but fails at $t=T^{\ast\ast}$.

   We will show that if $T^{\ast\ast}<T^{\ast},$ it will lead to a contradiction. In fact, for any $T\in[0,T^{\ast\ast})$, the definition of $T^{\ast\ast}$ yields that (\ref{a priori-conclusion}) holds for all $t\in[0,T]$. By virtue of Proposition \ref{uniform_estimates}, it holds that
   \begin{equation*}
    (1+T)\mathcal{L}(T) \le \frac{\delta}{2}.
   \end{equation*} By virtue of the continuity-in-time of the solution in $[0,T^{\ast})$ and $T^{\ast\ast}<T^{\ast},$ we let $T\rightarrow (T^{\ast\ast})^-$ and get
   \begin{equation*}
    (1+T^{\ast\ast})\mathcal{L}(T^{\ast\ast})  \le \delta,
   \end{equation*}which is a contradiction with the definition of $T^{\ast\ast}$. Therefore, $T^{\ast\ast}=T^{\ast}$ and (\ref{a priori-conclusion}) holds true.
  \end{proof}
  \subsection{Proof of Theorem \ref{global_existence_Lagrangian}}\label{subsection_6}
  In this subsection, we will finish the proof of Theorem \ref{global_existence_Lagrangian}. If $T^{\ast}=\infty$, the global solution is obtained. Suppose that $T^{\ast}<\infty$ and let $C_3$ be a generic positive constant depending on the initial data, $T^*$, and other known constants but independent of $t$. By \eqref{a priori-conclusion}, we deduce that for all $t\in [0,T^{\ast}),$
  \begin{equation} \label{contradiction_4_Lagrangian}
    \begin{split}
      \|v(t)\|_{H^3} + \|\widetilde{\eta}(t)\|_{H^4}\le C_3\varepsilon_1.
    \end{split}
  \end{equation}
  Taking  $T^*_1=T^{\ast}-\frac{T_1(C_3)}{2}\in [0,T^{\ast})$ as the initial time and applying  Proposition \ref{local_existence_1} with (\ref{contradiction_4_Lagrangian}), we can extend the existence time of the solution to
 $T^{\ast} + \frac{T_1}{2}.$ This contradicts the definition of $T^{\ast}$. Therefore, $T^{\ast} =\infty$, i.e., $(v,\widetilde{\eta})$ satisfying (\ref{a priori-conclusion}) for all $t\ge 0$ is the global solution of \eqref{compressible_elasticity_equation_5}. On the other hand, using \eqref{perturbation}, \eqref{perturbation_1} and $\eta = \widetilde{\eta} +y$, we have
 \begin{eqnarray*}\label{connection}
  \begin{cases}
  \xi(y,t)=(1+t)\widetilde{\eta}(y,t) +(1+t)y,\\
  V(y,t) = (1+t)v(y,t) + \frac{1}{(1+t)}\xi(y,t).
  \end{cases}
\end{eqnarray*}
It is easy to verify that $(\xi, V)$ is the global solution of \eqref{compressible_elasticity_equation_3} that satisfies
 \begin{eqnarray*}\label{Euler_part_3}
   \begin{cases}
     \|\xi(y,t) - (1+t)y\|_{H^3_y} = \|(1+t)\widetilde{\eta}(y,t)\|_{H^3_y}\le C\varepsilon_1(1+t)^{\frac{3}{2}},\\[1mm]
     \|\nabla_y \xi(y,t) - (1+t)I\|_{H^3_y} = \|(1+t)\nabla_y \widetilde{\eta}(y,t)\|_{H^3_y}\le C\varepsilon_1(1+t)^{\frac{1}{2}},\\[1mm]
     \|V(y,t) -y\|_{H^3_y} = \|(1+t)v(y,t) + \widetilde{\eta}(y,t)\|_{H^3_y} \le C(1+t)\|v(y,t)\|_{H^3_y} + C\|\widetilde{\eta}(y,t)\|_{H^3_y} \\
   \qquad\qquad\qquad\qquad\qquad\qquad\qquad\qquad\qquad\qquad \le C\varepsilon_1(1+t)^{\frac{1}{2}}.
   \end{cases}
 \end{eqnarray*}
{Thus \eqref{decay_estimates_Lagrangian} holds for all $t\ge 0$ and $\varepsilon_1\le \frac{1}{4\bar{C_1}}.$} The proof of the uniqueness is similar to that in Section \ref{subsection_3}. We omit it for brevity. The proof of Theorem  \ref{global_existence_Lagrangian} is complete.

\section{Global well-posedness in Eulerian coordinate}\label{section_4}
This section is devoted to the completion of the proof of Theorem \ref{global_existence}. Throughout  this section, C denotes a generic positive constant that depends on
some known constants, but independent of $\varepsilon$, $\varepsilon_1$ and $t$.
\subsection{Regularity estimates}\label{subsection_7}
With Theorem \ref{global_existence_Lagrangian}, we are going to construct a global-in-time solution to the system
\eqref{elasticity_equation_1} in Eulerian coordinate. For any   $(u_0, \xi_0)$ given in Theorem \ref{global_existence}, we define $V_0$ as follows
 \begin{equation}\label{V_0}
  V_0(y)=u_0(\xi_0(y)).
\end{equation}
{Combining \eqref{V_0} with the assumptions of Theorem \ref{global_existence}, we deduce
\begin{equation*}\label{1}
  \begin{split}
    \|V_0(y)-y\|_{H^3_y} + \|\xi_0(y) -y\|_{H^4_y}\le C\varepsilon  + C\varepsilon^3 \le C\varepsilon.
  \end{split}
\end{equation*}
Then, letting $\varepsilon = \frac{\varepsilon_1}{C}$ yields
\begin{equation}\label{2}
  \begin{split}
    \|V_0(y)-y\|_{H^3_y} + \|\xi_0(y) -y\|_{H^4_y}\le \varepsilon_1.
  \end{split}
\end{equation}
Substituting $V_0(y)$ and $\xi_0(y)$, that satisfy \eqref{V_0} and \eqref{2},  into Theorem \ref{global_existence_Lagrangian}, we obtain the existence of a unique global solution $(\xi,V)$ to the system \eqref{compressible_elasticity_equation_3}.

For any given $t\ge 0,$ by virtue of \eqref{decay_estimates_Lagrangian}$_2$, we deduce that the $C^1$ function $\xi(y,t): \mathbb{R}^d\to \mathbb{R}^d$ is  proper (i.e., the pre-image of every compact set is compact). Moreover, thanks to \eqref{decay_estimates_Lagrangian}$_3$ and the smallness of $\varepsilon$, we have
\begin{equation}\label{det}
  {\rm det}(\nabla_y \xi(y,t))> 0,
\end{equation}
for any $y\in \mathbb{R}^d$ and $t\ge 0.$

Since $\mathbb{R}^d$ is a smooth, connected $d$-dimensional manifold and is simply connected, 
we can apply the Hadamard's Global Inverse Function Theorem (see \cite[Theorem 6.2.8]{Krantz_2013}) to conclude that $\xi(\cdot,t):\mathbb R^d\to\mathbb R^d$ is a homeomorphism. This, combined with the $C^1$-regularity and strictly positive Jacobian determinant of $\xi(\cdot, t)$, implies
that $\xi(\cdot,t):\mathbb R^d\to\mathbb R^d$ is a $C^1$-diffeomorphism. Consequently, for any given $t\ge 0,$ the inverse function $y =\xi^{-1}(x,t): \mathbb{R}^d\to \mathbb{R}^d$ is well-defined.}

By substituting $y =\xi^{-1}(x,t)$ into the spatial variable in the solution stated in Theorem \ref{global_existence_Lagrangian}, we define $(\rho, u, F)$ in Eulerian coordinate in terms of $(\xi, V)$ as
\begin{eqnarray}\label{coordinate_transformation_1}
   \begin{cases}
    u(x,t) = V\Big(\xi^{-1}(x,t),t\Big),\\[1mm]
     F(x,t) = \nabla_y \xi\Big(\xi^{-1}(x,t),t\Big),\\[1mm]
    \rho(x,t) = \Big({\rm det}[\nabla_y \xi\big(\xi^{-1}(x,t),t\big)]\Big)^{-1}.
   \end{cases}
 \end{eqnarray}
Combining \eqref{coordinate_transformation_1} with Theorem \ref{global_existence_Lagrangian} and the fact that $(\xi, V)$ satisfies  system \eqref{compressible_elasticity_equation_3}, we conclude that $(\rho, u, F)$ is the unique global solution to \eqref{elasticity_equation_1}, see \cite{Jiang_2018} for instance.

Recall from \eqref{flow_map_2} that
\begin{equation*}
  B(y,t):= [(\nabla_{y} \xi)(y,t)^{-1}]^{\top},\,\,\,\,\,\,\,\,\,\,\,\,\,\,\,\,\,\,\,\,\,\,\,\,\,\,\,\,\mathcal{J}(y,t):= {\rm det}\nabla_{y} \xi(y,t).
\end{equation*}
For computational convenience, combining this definition with $y =\xi^{-1}(x,t)$, \eqref{decay_estimates_Lagrangian} and \eqref{coefficients_conditions}, we obtain
 \begin{eqnarray}\label{Euler_part_4}
   \begin{cases}
     \frac{1}{8}(1+t)^d\le \big|\mathcal{J}(\xi^{-1}(x,t),t)\big|\le 8(1+t)^d,\\[1mm]
     \frac{1}{8}(1+t)^{-d}\le \big|{\rm det}[\nabla_x \xi^{-1}(x,t)]\big| = \frac{1}{\big|\mathcal{J}(\xi^{-1}(x,t),t)\big|} \le 8(1+t)^{-d},\\[1mm]
     \|B(\xi^{-1}(x,t),t)\|_{L^\infty_x}\le \frac{C}{1+t}.
   \end{cases}
 \end{eqnarray}
Moreover, similar to \eqref{flow_map_2_1}, for any differentiable  function
$f(y,t)$ on $\mathbb{R}^d\times[0, +\infty)$, applying the chain rule together with $y =\xi^{-1}(x,t)$ yields
\begin{equation}\label{chain_rule}
   \begin{split}
    \nabla_{x_{i}} f(\xi^{-1}(x,t),t)& = \nabla_{y_{k}} f(\xi^{-1}(x,t),t) B_{ik}(\xi^{-1}(x,t),t)  \\ & = \nabla_{y_{k}} f(\xi^{-1}(x,t),t) \nabla_{x_{i}} \xi_{k}^{-1}(x,t),
   \end{split}
 \end{equation}
 where we have used the fact  $\nabla_{x} \xi^{-1}(x,t) \nabla_{y} \xi\big(\xi^{-1}(x,t),t\big)= I$ and that
 $$B_{ik}(\xi^{-1}(x,t),t) =\big(\nabla_{x} \xi^{-1}(x,t)\big)_{ik}^{\top} =  \nabla_{x_{i}} \xi_{k}^{-1}(x,t).$$
Using $y =\xi^{-1}(x,t)$,  \eqref{coordinate_transformation_1}$_{1}$, \eqref{Euler_part_4}$_{2,3}$ and \eqref{chain_rule}, we obtain
 \begin{equation}\label{Euler_part_12}
     \begin{split}
        &\int_{\mathbb{R}^d}  |V(y,t) - \frac{\xi(y,t)}{1+t}|^2 \,dy\\ 
        &=\int_{\mathbb{R}^d}  \big|V\big((\xi^{-1}(x,t),t)\big) - \frac{\xi\big((\xi^{-1}(x,t),t)\big)}{1+t}\big|^2 \big|{\rm det} [\nabla_x \xi^{-1}(x,t)]\big| \,dx\\
       &= \int_{\mathbb{R}^d}\big|u\big(x,t\big) - \frac{x}{1+t}\big|^2 \big|{\rm det} [\nabla_x \xi^{-1}(x,t)]\big|\,dx\\
      &\ge\frac{1}{8}(1+t)^{-d}\|u(x,t) - \frac{x}{1+t}\|_{L^2_x}^2
     \end{split}
   \end{equation}
and
\begin{equation}\label{Euler_part_13}
     \begin{split}
        &\int_{\mathbb{R}^d}  |\nabla_y\big(V(y,t) - \frac{\xi(y,t)}{1+t}\big)|^2 \,dy\\ 
        &\ge\frac{1}{|B|^2}\int_{\mathbb{R}^d}  |B(y,t)\big(\nabla_y V(y,t) - \frac{\nabla_y \xi(y,t)}{1+t}\big)|^2\,dy \\ 
        &\ge\frac{1}{C}(1+t)^2\int_{\mathbb{R}^d}  \big|B(\xi^{-1}(x,t),t)\nabla_y V\big((\xi^{-1}(x,t),t)\big) - \frac{I}{1+t}\big|^2 \big|{\rm det} [\nabla_x \xi^{-1}(x,t)]\big| \,dx\\
       &\ge \frac{1}{C}(1+t)^2\int_{\mathbb{R}^d}\big|\nabla_x \big(u(x,t) - \frac{x}{1+t}\big)\big|^2 \big|{\rm det} [\nabla_x \xi^{-1}(x,t)]\big|\,dx\\
      &\ge\frac{1}{8C}(1+t)^{2-d}\|\nabla_x \big(u(x,t) - \frac{x}{1+t}\big)\|_{L^2_x}^2,
     \end{split}
   \end{equation}
where we have used the fact that $B_{is}(y,t)\nabla_{y_s}\xi_{j}(y,t) = \delta_{ij}.$

Similarly, we have
 \begin{equation}\label{Euler_part_8}
   \begin{split}
 \int_{\mathbb{R}^d}  |B\nabla_y(B\nabla_y V)|^2 \,dy\ge \,\frac{1}{C} (1+t)^{-d}\|\nabla^2_x \big(u(x,t) - \frac{x}{1+t}\big)\|_{L^2_x}^2\,\,\,\,\,\,\,\,\,\,\,\,\,\,\,\,
   \end{split}
  \end{equation}
and
  \begin{align}\label{Euler_part_9}
     \int_{\mathbb{R}^d}  \big|B\nabla_y[B\nabla_y(B\nabla_y V)]\big|^2 \,dy\ge \,\frac{1}{C}(1+t)^{-d}\|\nabla^3_x \big(u(x,t) - \frac{x}{1+t}\big)\|_{L^2_x}^2.
  \end{align}
 For computational simplicity, by using $\nabla_y B_{ij}= -B_{il}\partial_l\nabla_y\xi_{k}B_{kj}$ and $\nabla_y \mathcal{J}= \mathcal{J} B_{ij}\partial_j\nabla_y \xi_{i}$ in \eqref{flow_map_5_1}, we conclude that
 \begin{eqnarray}\label{Euler_part_4_1}
  \begin{cases}
 |\nabla_y B| \le C |B|^2|\nabla^2_y \xi|,\\
 |\nabla_y^2 B| \le C|B|^2|\nabla^3_y \xi| + C|B|^3|\nabla^2_y\xi|^2,\\
  |\nabla_y \mathcal{J}| \le C|B||\mathcal{J}||\nabla^2_y \xi|,\\
  |\nabla_y^2 \mathcal{J}|\le C|B||\mathcal{J}||\nabla^3_y \xi|+C|\mathcal{J}||B|^2|\nabla^2_y\xi|^2,\\
  |\nabla_y^3 \mathcal{J}|\le C |B||\mathcal{J}||\nabla^4_y \xi|+C|\mathcal{J}||B|^2|\nabla^2_y\xi||\nabla^3_y\xi|+C|\mathcal{J}||B|^3|\nabla^2_y\xi|^3.
  \end{cases}
\end{eqnarray}
Using \eqref{decay_estimates_Lagrangian}, \eqref{Euler_part_12}-\eqref{Euler_part_9}, \eqref{Euler_part_4_1}, $\varepsilon = \frac{\varepsilon_1}{C}$ and $\nabla_y B_{ij}= -B_{il}\partial_l\nabla_y\xi_{k}B_{kj}$ in \eqref{flow_map_5_1}, we have
 \begin{equation}\label{Euler_part_ww_1}
   \begin{split}
    \|u(x,t) - \frac{x}{1+t}\|_{L^2_x}^2 \le\,& C(1+t)^d\int_{\mathbb{R}^d}  |V(y,t) - \frac{\xi(y,t)}{1+t}|^2 \,dy\\\le\,& C(1+t)^d\int_{\mathbb{R}^d} \Big(|V - y|^2 + |\frac{\xi}{1+t} - y|^2 \Big) \,dy\\ \le \,& C\varepsilon^2 (1+t)^{d+1},
   \end{split}
 \end{equation}
  \begin{equation}\label{Euler_part_ww_2}
   \begin{split}
    \|\nabla_x \big(u(x,t) - \frac{x}{1+t}\big)\|_{L^2_x}^2 \le\,& C(1+t)^{d-2}\int_{\mathbb{R}^d}  \big|\nabla_y\big(V(y,t) - \frac{\xi(y,t)}{1+t}\big)\big|^2\,dy\\\le\,& C(1+t)^{d-2}\int_{\mathbb{R}^d} \Big(|\nabla_y V - I|^2 + |\frac{\nabla_y \xi}{1+t} - I|^2 \Big)\,dy\\ \le\,&  C\varepsilon^2 (1+t)^{d-1},
   \end{split}
 \end{equation}
  \begin{equation}\label{Euler_part_ww_3}
   \begin{split}
    \|\nabla_x^2 \big(u(x,t) - \frac{x}{1+t}\big)\|_{L^2_x}^2 \le\,& C(1+t)^d \int_{\mathbb{R}^d}  \big|B\nabla_y(B\nabla_y V)\big|^2 \,dy\\\le\,& C(1+t)^d\Big(\int_{\mathbb{R}^d}  |B B\nabla^2_y V|^2 \,dy + \int_{\mathbb{R}^d}  |B \nabla_y B\nabla_y V|^2 \,dy\Big)\\\le\,& C\varepsilon^2(1+t)^{d-3} + C\varepsilon^2(1+t)^{d-4}\\\le\,& C\varepsilon^2(1+t)^{d-3}
   \end{split}
 \end{equation}
 and
  \begin{equation}\label{Euler_part_ww_4}
   \begin{split}
    &\|\nabla_x^3 \big(u(x,t) - \frac{x}{1+t}\big)\|_{L^2_x}^2 \\
    &\le C(1+t)^d \int_{\mathbb{R}^d}  \big|B\nabla_y[B\nabla_y(B\nabla_y V)]\big|^2 \,dy\\
    &\le C(1+t)^d\Big[\int_{\mathbb{R}^d}  |B\nabla_yB\nabla_yB\nabla_y V|^2 \,dy + \int_{\mathbb{R}^d}  |BB\nabla_y^2 B\nabla_y V|^2 \,dy\\
    &+ \int_{\mathbb{R}^d}  |BB\nabla_yB\nabla_y^2 V|^2 \,dy + \int_{\mathbb{R}^d}  |BBB\nabla_y^3 V|^2 \,dy\Big]\\
    &\le C\varepsilon^2(1+t)^{d-7} + C\varepsilon^2(1+t)^{d-6}+ C\varepsilon^2(1+t)^{d-6} + C\varepsilon^2(1+t)^{d-5}\\
    &\le C\varepsilon^2(1+t)^{d-5}.
   \end{split}
 \end{equation}
Combining $y =\xi^{-1}(x,t)$ with \eqref{coordinate_transformation_1}$_{2}$, \eqref{Euler_part_4}$_{2,3}$  and \eqref{chain_rule}, we derive
 \begin{equation}\label{Euler_part_5}
   \begin{split}
    &\int_{\mathbb{R}^d} |\nabla_y \xi(y,t) - (1+t)I|^2\,dy\\ 
    &= \int_{\mathbb{R}^d} |\nabla_y \xi(\xi^{-1}(x,t),t) - (1+t)I|^2 \big|{\rm det} [\nabla_x \xi^{-1}(x,t)]\big|\,dx\\ 
    &=
      \int_{\mathbb{R}^d} |F(x,t) - (1+t)I|^2 \big|{\rm det} [\nabla_x \xi^{-1}(x,t)]\big|\,dx\\
      &\ge \frac{1}{8}(1+t)^{-d}\|F(x,t) - (1 + t)I\|_{L^2_x}^2
   \end{split}
 \end{equation}
 and
 \begin{equation}\label{Euler_part_6}
   \begin{split}
    &\int_{\mathbb{R}^d} |\nabla_{y}(\nabla_y \xi(y,t) - (1+t)I)|^2\,dy\\ 
    &\ge
    \frac{1}{|B|^2} \int_{\mathbb{R}^d} |B(y,t)\nabla_{y}(\nabla_y \xi(y,t) - (1+t)I)|^2\,dy\\ 
    &\ge
    \frac{1}{C}(1+t)^2 \int_{\mathbb{R}^d} |B(\xi^{-1}(x,t),t) \nabla_{y}\big(\nabla_y \xi(\xi^{-1}(x,t),t) - (1+t)I\big)|^2 \big|{\rm det} [\nabla_x \xi^{-1}(x,t)]\big|\,dx\\ 
    &\ge
    \frac{1}{C}(1+t)^2
      \int_{\mathbb{R}^d} [\nabla_{x}\big(F(x,t) - (1+t)I\big)]^2 \big|{\rm det} [\nabla_x \xi^{-1}(x,t)]\big|\,dx \\
      &\ge\frac{1}{8C}(1+t)^{2-d}\|\nabla_{x} \big(F(x,t) - (1 + t)I\big)\|_{L^2_x}^2.
   \end{split}
 \end{equation}
Similar to the computations of \eqref{Euler_part_5} and \eqref{Euler_part_6}, using $y =\xi^{-1}(x,t)$ and \eqref{coordinate_transformation_1}-\eqref{chain_rule}, we derive
 \begin{equation}\label{Euler_part_6_1}
  \begin{split}
    \int_{\mathbb{R}^d}  |B\nabla_y(B\nabla_y^2 \xi)|^2\,dy \ge \,\frac{1}{C}(1+t)^{-d}\|\nabla_x^2 \big(F(x,t) - (1+t)I\big)\|_{L^2_x}^2\,\,\,\,\,\,\,\,\,\,\,\,\,\,
  \end{split}
\end{equation}
and
 \begin{equation}\label{Euler_part_7}
   \begin{split}
    \int_{\mathbb{R}^d}  |B\nabla_y[B\nabla_y(B\nabla_y^2 \xi)]|^2 \,dy \ge \,\frac{1}{C}(1+t)^{-d}\|\nabla_x^3 \big(F(x,t) - (1+t)I\big)\|_{L^2_x}^2.
   \end{split}
 \end{equation}
Using \eqref{decay_estimates_Lagrangian}, \eqref{Euler_part_5}-\eqref{Euler_part_4_1}, {$\varepsilon = \frac{\varepsilon_1}{C}$} and $\nabla_y B_{ij}= -B_{il}\partial_l\nabla_y\xi_{k}B_{kj}$ in \eqref{flow_map_5_1}, we have
 \begin{equation}\label{Euler_part_ee_1}
   \begin{split}
    \|F(x,t)-(1+t)I\|_{L^2_x}^2 \le\,& C(1+t)^d\int_{\mathbb{R}^d} |\nabla_y \xi(y,t) - (1+t)I|^2\,dy\le C\varepsilon (1+t)^{d+1},
   \end{split}
 \end{equation}
  \begin{equation}\label{Euler_part_ee_2}
   \begin{split}
    \|\nabla_x \big(F(x,t)-(1+t)I\big)\|_{L^2_x}^2 \le\,& C(1+t)^{d-2}\int_{\mathbb{R}^d} |\nabla_y(\nabla_y \xi(y,t) - (1+t)I)|^2\,dy\\ \le&\, C\varepsilon^2 (1+t)^{d-1},
   \end{split}
 \end{equation}
  \begin{equation}\label{Euler_part_ee_3}
   \begin{split}
    \|\nabla_x^2 \big(F(x,t)-(1+t)I\big)\|_{L^2_x}^2 \le\,& C(1+t)^d\int_{\mathbb{R}^d}  |B\nabla_y(B\nabla_y^2 \xi)|^2\,dy\\\le\,& C(1+t)^d\Big(\int_{\mathbb{R}^d}  |B\nabla_yB\nabla_y^2 \xi|^2 \,dy + \int_{\mathbb{R}^d}  |BB\nabla_y^3 \xi|^2 \,dy\Big)\\\le\,& C\varepsilon^2(1+t)^{d-4} + C\varepsilon^2(1+t)^{d-3}\\\le\,& C\varepsilon^2(1+t)^{d-3}
   \end{split}
 \end{equation}
 and
  \begin{equation}\label{Euler_part_ee_4}
   \begin{split}
    &\|\nabla_x^3 \big(F(x,t)-(1+t)I\big)\|_{L^2_x}^2 \\
    &\le C(1+t)^d \int_{\mathbb{R}^d}  \big|B\nabla_y[B\nabla_y(B\nabla_y^2 \xi)]\big|^2 \,dy\\
    &\le C(1+t)^d\Big[\int_{\mathbb{R}^d}  |B\nabla_yB\nabla_yB\nabla_y^2 \xi|^2 \,dy + \int_{\mathbb{R}^d}  |BB\nabla_y^2 B\nabla_y^2 \xi|^2 \,dy\\
    &+ \int_{\mathbb{R}^d}  |BB\nabla_yB\nabla_y^3 \xi|^2 \,dy + \int_{\mathbb{R}^d}  |BBB\nabla_y^4 \xi|^2 \,dy\Big]\\
    &\le C\varepsilon^2(1+t)^{d-7} + C\varepsilon^2(1+t)^{d-6}+ C\varepsilon^2(1+t)^{d-6} + C\varepsilon^2(1+t)^{d-5}\\
    &\le C\varepsilon^2(1+t)^{d-5}.
   \end{split}
 \end{equation}
Combining $y =\xi^{-1}(x,t)$ with \eqref{coordinate_transformation_1}$_{3}$ and \eqref{Euler_part_4}$_{2}$, we obtain
  \begin{equation}\label{Euler_part_nn_1}
     \begin{split}
        &\int_{\mathbb{R}^d} \Big|\frac{1}{{\rm det}[\nabla_y \xi(y,t)]}- \frac{1}{ (1+t)^d}\Big|^2\,dy\\ 
        &=\int_{\mathbb{R}^d} \Big|\frac{1}{{\rm det}[\nabla_y \xi(\xi^{-1}(x,t),t)]}- \frac{1}{ (1+t)^d}\Big|^2\Big|{\rm det} [\nabla_x \xi^{-1}(x,t)]\Big| \,dx\\
       &=\int_{\mathbb{R}^d}\big|\rho(x,t) - \frac{1}{(1+t)^d}\big|^2 \big|{\rm det} [\nabla_x \xi^{-1}(x,t)]\big|\,dx\\
      &\ge\frac{1}{8}(1+t)^{-d}\|\rho(x,t) -\frac{1}{(1+t)^d}\|_{L^2_x}^2.
     \end{split}
   \end{equation}
 Similarly, using $y =\xi^{-1}(x,t)$, \eqref{coordinate_transformation_1}$_{3}$, \eqref{Euler_part_4}$_{2}$ and \eqref{chain_rule}, we derive
 \begin{equation}\label{Euler_part_nn_2}
     \begin{split}
      \int_{\mathbb{R}^d}  \Big|B\nabla_y \Big(\frac{1}{{\rm det}[\nabla_y \xi]}\Big)\Big|^2\,dy
      \ge \,\frac{1}{C}(1+t)^{-d}\|\nabla_x \big(\rho(x,t) -\frac{1}{(1+t)^d}\big)\|_{L^2_x}^2,\,\,\,\,\,\,\,\,\,\,\,\,\,\,\,\,\,\,\,\,\,\,\,\,\,\,\,\,\,\,\,\,\,
     \end{split}
   \end{equation}
 \begin{equation}\label{Euler_part_nn_3}
     \begin{split}
       \int_{\mathbb{R}^d}  \Big|B\nabla_y\Big(B\nabla_y \big(\frac{1}{{\rm det}[\nabla_y \xi]}\big)\Big)\Big|^2\,dy
      \ge \,\frac{1}{C}(1+t)^{-d}\|\nabla_x^2 \big(\rho(x,t) -\frac{1}{(1+t)^d}\big)\|_{L^2_x}^2\,\,\,\,\,\,\,\,\,\,\,\,\,\,\,\,\,
     \end{split}
   \end{equation}
   and
   \begin{equation}\label{Euler_part_nn_4}
     \begin{split}
       \int_{\mathbb{R}^d}  \Big|B\nabla_y\Big[B\nabla_y\Big(B\nabla_y \big(\frac{1}{{\rm det}[\nabla_y \xi]}\big)\Big)\Big]\Big|^2\,dy
      \ge \,\frac{1}{C}(1+t)^{-d}\|\nabla_x^3 \big(\rho(x,t) -\frac{1}{(1+t)^d}\big)\|_{L^2_x}^2.
     \end{split}
   \end{equation}
By \eqref{Euler_part_nn_1}, we derive
 \begin{equation}\label{Euler_part_14}
     \begin{split}
        \|\rho(x,t) -\frac{1}{(1+t)^d}\|_{L^2_x}^2 &\le C{(1+t)^d}  \int_{\mathbb{R}^d} \Big|\frac{1}{(1+t)^d \mathcal{J}}\Big((1+t)^d - {\rm det}[\nabla_y \xi]\Big)\Big|^2\,dy\\
       &\le \frac{C}{(1+t)^{3d}}\int_{\mathbb{R}^d} \Big|{\rm det}[\nabla_y \xi]I - (1+t)^{d}I\Big|^2\,dy.
     \end{split}
   \end{equation}
Using the identity ${\rm det}[\nabla_y \xi] I = \nabla_y \xi {\rm adj}[\nabla_y \xi]$, we decompose:
 \begin{equation}\label{Euler_part_15}
   \begin{split}
     {\rm det}[\nabla_y \xi] I- (1+t)^dI = \big[\nabla_y \xi - (1+t)I\big]{\rm adj}[\nabla_y \xi] + (1+t)\big[{\rm adj}[\nabla_y \xi] - (1+t)^{d-1}I\big].
   \end{split}
 \end{equation}
For the first term, estimates \eqref{decay_estimates_Lagrangian}$_3$ and {$\varepsilon = \frac{\varepsilon_1}{C}$} yield
 \begin{equation}\label{Euler_part_16}
   \begin{split}
     \int_{\mathbb{R}^d} \Big|\big[\nabla_y \xi - (1+t)I\big]{\rm adj}[\nabla_y \xi]\Big|^2\,dy \le C \varepsilon^2 (1+t)^{2d-1},
   \end{split}
 \end{equation}
 where we have used the fact that $\big|{\rm adj}[\nabla_y \xi]\big|\le C|\nabla_y \xi|^{d-1} \le C(1+t)^{d-1}$ in \eqref{decay_estimates_Lagrangian}$_3.$

For the second term, we similarly obtain
 \begin{equation}\label{Euler_part_17}
   \begin{split}
     &\int_{\mathbb{R}^d} \Big|{\rm adj}[\nabla_y \xi] - (1+t)^{d-1}I\Big|^2\,dy\\ &\le C \Big(\|\nabla_y \xi\|_{L^\infty_y}^{2(d-2)} + (1+t)^{2(d-2)}\Big)\|\nabla_y \xi -(1+t)I\|_{L^2_y}^2\\ & \le C\varepsilon^2 (1+t)^{2d-3}.
   \end{split}
 \end{equation}
Substituting \eqref{Euler_part_16} and \eqref{Euler_part_17} into \eqref{Euler_part_14} and \eqref{Euler_part_15} gives
 \begin{equation}\label{Euler_part_18}
   \begin{split}
     \|\rho(x,t) -\frac{1}{(1+t)^d}\|_{L^2_x}^2 \le C\varepsilon^2 (1+t)^{-d-1}.
   \end{split}
 \end{equation}
 Next, we consider the estimate of $\|\nabla \big(\rho(x,t) -\frac{1}{(1+t)^d}\big)\|_{L^2_x}$ and $\|\nabla^2 \big(\rho(x,t) -\frac{1}{(1+t)^d}\big)\|_{L^2_x}$. Combining \eqref{decay_estimates_Lagrangian}, \eqref{Euler_part_4_1}, \eqref{Euler_part_nn_2}-\eqref{Euler_part_nn_3}, {$\varepsilon = \frac{\varepsilon_1}{C}$} with $\nabla_y B_{ij}= -B_{il}\partial_l\nabla_y\xi_{k}B_{kj}$ and $\nabla_y \mathcal{J}= \mathcal{J} B_{ij}\partial_j\nabla_y \xi_{i}$ in \eqref{flow_map_5_1}, we have
 \begin{equation}\label{Euler_part_10}
   \begin{split}
    &\|\nabla_x \big(\rho(x,t) -\frac{1}{(1+t)^d}\big)\|_{L^2_x}^2 \\ &\le C(1+t)^d \int_{\mathbb{R}^d}  |B\nabla_y(\mathcal{J}^{-1})|^2\,dy\\
     &\le C(1+t)^d\int_{\mathbb{R}^d}  |B\mathcal{J}^{-2}\nabla_y\mathcal{J}|^2\,dy\\
     &\le C(1+t)^d|B|^4 |\mathcal{J}|^{-2} \int_{\mathbb{R}^d}  |\nabla^2_y \xi|^2\,dy\\
     &\le C \varepsilon^2 (1+t)^{-d-3}
   \end{split}
 \end{equation}
 and
   \begin{equation}\label{Euler_part_11}
   \begin{split}
    &\|\nabla^2_x \big(\rho(x,t) -\frac{1}{(1+t)^d}\big)\|_{L^2_x}^2\\ 
    &\le C(1+t)^d\int_{\mathbb{R}^d}  \big|B\nabla_y[B\nabla_y(\mathcal{J}^{-1})]\big|^2\,dy \\
   &\le C(1+t)^d\Big[\int_{\mathbb{R}^d}  |B\nabla_y B\mathcal{J}^{-2}\nabla_y\mathcal{J}|^2 \,dy + \int_{\mathbb{R}^d}  |BB \mathcal{J}^{-3}\nabla_y\mathcal{J}\nabla_y\mathcal{J}|^2 \,dy\\ 
   & +\int_{\mathbb{R}^d}  |BB\mathcal{J}^{-2}\nabla_y^2\mathcal{J}|^2 \,dy\Big]\\
    &\le C \varepsilon^2 (1+t)^{-d-6} + C \varepsilon^2 (1+t)^{-d-6} +C \varepsilon^2 (1+t)^{-d-5}  \\  
    &\le C \varepsilon^2 (1+t)^{-d-5}.
  \end{split}
 \end{equation}
  Finally, we estimate $\|\nabla^3_x \big(\rho(x,t) -\frac{1}{(1+t)^d}\big)\|_{L^2_x}$. Combining \eqref{decay_estimates_Lagrangian}, \eqref{Euler_part_4_1}, \eqref{Euler_part_nn_4}, $\varepsilon = \frac{\varepsilon_1}{C}$ with $\nabla_y B_{ij}= -B_{il}\partial_l\nabla_y\xi_{k}B_{kj}$ and $\nabla_y \mathcal{J}= \mathcal{J} B_{ij}\partial_j\nabla_y \xi_{i}$ in \eqref{flow_map_5_1}, we have
  \begin{align}\label{Euler_part_11_1}
        &\notag\|\nabla^3_x \big(\rho(x,t) -\frac{1}{(1+t)^d}\big)\|_{L^2_x}^2\\ \notag &\le C(1+t)^d\int_{\mathbb{R}^d}  \Big|B\nabla_y\big\{B\nabla_y\big[B\nabla_y(\mathcal{J}^{-1})\big]\big\}\Big|^2\,dy\\ &
    \le C(1+t)^d\Big[\int_{\mathbb{R}^d}  |B\nabla_y B\nabla_y B\mathcal{J}^{-2}\nabla_y\mathcal{J}|^2 \,dy + \int_{\mathbb{R}^d}  |BB \nabla_y^2 B\mathcal{J}^{-2}\nabla_y\mathcal{J}|^2 \,dy\notag\\ & +\int_{\mathbb{R}^d}  |BB\nabla_y B\mathcal{J}^{-3}\nabla_y\mathcal{J}\nabla_y\mathcal{J}|^2 \,dy +\int_{\mathbb{R}^d}  |BB\nabla_y B\mathcal{J}^{-2}\nabla_y^2\mathcal{J}|^2 \,dy\\ & +\int_{\mathbb{R}^d}  |BBB\mathcal{J}^{-4}\nabla_y\mathcal{J}\nabla_y\mathcal{J}\nabla_y\mathcal{J}|^2 \,dy +\int_{\mathbb{R}^d}  |BBB\mathcal{J}^{-3}\nabla_y\mathcal{J}\nabla_y^2\mathcal{J}|^2 \,dy\notag\\ & \notag +\int_{\mathbb{R}^d}  |BBB\mathcal{J}^{-2}\nabla_y^3\mathcal{J}|^2\,dy\Big]\\ &\notag
    \le\, C \varepsilon^2 (1+t)^{-d-9} + C \varepsilon^2 (1+t)^{-d-8} + C \varepsilon^2 (1+t)^{-d-9}+ C \varepsilon^2 (1+t)^{-d-8} \notag \\ &+ C \varepsilon^2 (1+t)^{-d-9} + C \varepsilon^2 (1+t)^{-d-8} + C \varepsilon^2 (1+t)^{-d-7}\notag\\ &\le \,C \varepsilon^2 (1+t)^{-d-7}.\notag
\end{align}
 From  \eqref{Euler_part_ww_1}-\eqref{Euler_part_ww_4}, \eqref{Euler_part_ee_1}-\eqref{Euler_part_ee_4}, and \eqref{Euler_part_18}-\eqref{Euler_part_11_1}, we conclude that
 \begin{eqnarray*}\label{decay_1}
   \begin{cases}
     \|\nabla_x^i \big(\rho(x,t) -\frac{1}{(1+t)^d}\big)\|_{L^2_x}\le C\varepsilon (1+t)^{-\frac{d}{2}-\frac{1}{2}-i},\\[1mm]
     \|\nabla_x^i\big(u(x,t)-\frac{x}{1+t}, F(x,t) -(1+t)I\big)\|_{L^2_x} \le C\varepsilon(1+t)^{\frac{d}{2}+\frac{1}{2}-i}
   \end{cases}
 \end{eqnarray*}
 for all $t \ge 0$ and $i=0,1,2,3$. The proof of Theorem \ref{global_existence} is complete.

\section*{Acknowledgments}X. Hu was partially supported by the RFS Grant and GRF Grant from the Research Grants Council
(Project Numbers PolyU 11302523, 11300420, and 11302021). H. Wen was supported by the National Natural Science Foundation of China $\#12471209$ and by the Guangzhou Basic and Applied Research Projects $\#SL2024A04J01206$.  C. Wang
was partially supported by NSF grants 2101224 and 2453789, and Simons Travel Grant TSM-00007723.

\end{document}